\documentclass[11pt,a4paper]{amsart}
\usepackage[utf8]{inputenc}
\usepackage[english]{babel}
\usepackage[dvipsnames]{xcolor}
\usepackage[all]{xy}
\usepackage{tikz}
\usepackage{tikz-cd}
\usepackage{soul,color}
\usepackage[colorinlistoftodos]{todonotes}
\usepackage[width=\textwidth,font=footnotesize,labelfont=bf, aboveskip=30pt]{caption}
\usepackage[symbol]{footmisc}

\newcommand*{\fullref}[1]{\hyperref[{#1}]{\ref*{#1}. \nameref*{#1}}} 
\newcommand{\sO}{\mathcal{O}}

\newcommand{\sP}{\mathscr{P}}

\newcommand{\Num}{\operatorname{Num}}
\newcommand{\set}[1]{\{{#1}\}}
\newcommand{\Ker}{\operatorname{Ker}}
\usepackage[osf]{mathpazo}
\usepackage{charter}
\usepackage[a4paper, top=3.2cm, bottom=3.2cm,left=2.8cm, right=2.8cm, heightrounded,bindingoffset=0mm]{geometry}

\usepackage{hyperref}
\hypersetup{bookmarksnumbered=true, linkcolor=due, citecolor=Green, urlcolor=black,}

\usepackage{fancyhdr}

\usepackage{amsfonts}
\usepackage{amssymb}
\usepackage{amsthm}
\usepackage{amscd}
\usepackage{amsmath}
\usepackage{mathtools}
\usepackage{mathrsfs}
\usepackage{float}

\usepackage{color}	
\definecolor{due}{RGB}{0,76,147}

\usepackage{graphicx}	
\usepackage{multicol}	
\usepackage{wrapfig}
\usepackage[english]{babel} 
\usepackage[utf8]{inputenc}
\usepackage[T1]{fontenc}
\usepackage{enumitem}

\usepackage{longtable}
\usepackage{cite}

\theoremstyle{definition}
\newtheorem{defi}{Definition}[section]
\theoremstyle{plain}
\newtheorem{thm}[defi]{Theorem}

\newtheorem{introthm}{Theorem}[section]

\newtheorem{prop}[defi]{Proposition}
\newtheorem{cor}[defi]{Corollary}
\newtheorem{lemma}[defi]{Lemma}
\theoremstyle{remark}
\newtheorem*{question}{Question}

\newtheorem{rmk}[defi]{Remark}
\newtheorem{ex}[defi]{Example}
\newtheorem{nota}[defi]{Notation}
\newtheorem*{notation}{Notation}

\theoremstyle{definition}

\newtheorem*{ack}{Acknowledgement}

\newcommand{\Q}{{\mathbb Q}}

\newcommand{\Z}{{\mathbb Z}}

\newcommand{\Pic}{\operatorname{Pic}}

\newcommand{\End}{\operatorname{End}}
\newcommand{\Hom}{\operatorname{Hom}}

\newcommand{\ord}{\operatorname{ord}}

  \makeatletter
\newcommand{\xdashrightarrow}[2][]{\ext@arrow 0359\rightarrowfill@@{#1}{#2}}
\newcommand{\xdashleftarrow}[2][]{\ext@arrow 3095\leftarrowfill@@{#1}{#2}}
\newcommand{\xdashleftrightarrow}[2][]{\ext@arrow 3359\leftrightarrowfill@@{#1}{#2}}
\def\rightarrowfill@@{\arrowfill@@\relax\relbar\rightarrow}
\def\leftarrowfill@@{\arrowfill@@\leftarrow\relbar\relax}
\def\leftrightarrowfill@@{\arrowfill@@\leftarrow\relbar\rightarrow}
\def\arrowfill@@#1#2#3#4{%
  $\m@th\thickmuskip0mu\medmuskip\thickmuskip\thinmuskip\thickmuskip
   \relax#4#1
   \xleaders\hbox{$#4#2$}\hfill
   #3$%
}
\makeatother

\numberwithin{equation}{section}
\begin{document}
	\title[Brauer group of bielliptic surfaces]{On the Brauer group of bielliptic surfaces}
	\author{Eugenia Ferrari}
	\address{Department of Mathematics\\ University of Bergen, All\'egaten 41, Bergen, Norway}
\email{\url{eugenia.ferrari@uib.no}, \url{tirabassi@math.su.se}, \url{magnus.vodrup@uib.no}}
\author{Sofia Tirabassi}
\address{Department of Mathematics\\ Stockholm University, Kraftriket hus 6, Stockholm, Sweden}
\email{\url{jonasb@math.su.se}, \url{tirabassi@math.su.se}}
\author{Magnus Vodrup}


\clearpage\maketitle
\vspace{-.8cm}\begin{center}
    {\scshape with an appendix by Jonas Bergstr\"om and Sofia Tirabassi}
\end{center}
\thispagestyle{empty}
\vspace{.5cm}
\begin{abstract}
We provide explicit generators of the torsion of the second cohomology of bielliptic surfaces, and we use these to study the pullback map between the Brauer group of a bielliptic surface and that of its canonical cover. We completely describe the loci where this map is injective, trivial, or neither one of them. Explicit examples are constructed for all possible different behaviors of this morphism.
\end{abstract}

\setcounter{tocdepth}{1}
\tableofcontents
\section{Introduction}
Given a smooth complex projective variety $Z$ its (cohomological) Brauer group is defined as $\operatorname{Br}(Z):=H_{\text{\'et}}^2(Z,\sO_Z^*)_{\text{tor}}$. A morphism of projective varieties $f:Z\rightarrow Y$ induces, via pullbacks, a homomorphism $f_{\operatorname{Br}}:\operatorname{Br}(Y)\rightarrow\operatorname{Br}(Z)$, which we call the \emph{Brauer map induced by $f$}. In \cite{Be2009} Beauville studies this map in the case of a complex Enriques surface $S$ and that of its K3 canonical cover $\pi:X\rightarrow S$. More precisely the author of \cite{Be2009} identifies the locus in the moduli space of Enriques surfaces in which $\pi_{\operatorname{Br}}$ is not injective (and therefore trivial). In this paper we carry out a similar investigation in the case of bielliptic surfaces.\par
A bielliptic surface is constructed by taking the quotient of a product of ellipic curves $A\times B$ by the action of a finite group $G$. They were classified in 7 different types by Bagnera--De Franchis (see \cite{BdF}, or \cite{Suwa1970OnHS} for a presentation in a more modern language), as illustrated in Table \ref{table:ta}.
\begin{table}[hb]
\renewcommand{\arraystretch}{1.3}
\centering
\begin{tabular}{l*{4}{c}}
\hline
Type & $G$  & Order of $\omega_S$ in $\mathrm{Pic}(S)$ & $H^2(S,\Z)_{\text{tor}}$\\
\hline
1 & $\Z/2\Z$  & 2& $\Z/2\Z\times\Z/2\Z$\\

2 & $\Z/2\Z \times \Z/2\Z$ & 2 &$\Z/2\Z$ \\

3  & $\Z/4\Z$  & 4 &$\Z/2\Z$\\

4  & $\Z/4\Z \times \Z/2\Z$  & 4 & 0\\

5  & $\Z/3\Z$  & 3 &$\Z/3\Z$\\

6 & $\Z/3\Z \times \Z/3\Z$ & 3 & 0\\

7 & $\Z/6\Z$ & 6 &0 \\ 
\hline
\end{tabular}
\caption{Types of bielliptic surfaces and torsion of their second cohomology.}\label{table:ta}
\vspace{-1.3cm}
\end{table}
Since the canonical bundle of a bielliptic surface $S$ is a torsion element in $\Pic(S)$, it can be used to define an \'etale cyclic cover $\pi:X\rightarrow S$, where $X$ is an abelian variety isogenous to $A\times B$. We then obtain a homomorphism between the respective Brauer groups: $\pi_{\operatorname{Br}}:\operatorname{Br}(S)\rightarrow\operatorname{Br}(X).$  A very natural question is the following.
\begin{question}
When is $\pi_{\text{Br}}$ injective? When is it trivial?
\end{question}

As for Enriques surfaces, using the long exact exponential sequence and Poincaré duality, we have a non canonical isomorphism
$$\operatorname{Br}(S)\simeq H^2(X,\Z)_{\text{tor}},$$ so  from the fourth column of Table \ref{table:ta}, we easily see that this map is trivial when $S$ is of type 4, 6 or 7.  Thus we will limit ourselves to surfaces of type 1, 2, 3, and 5. We will find that the behavior of   the Brauer map depends heavily  on the geometry of the bielliptic surface $S$.\par

Bielliptic surfaces of type 2 and 3 admit (by a construction which is written explicitly only in \cite{Howie}) a degree 2 \'etale cover $\tilde{\pi}:\tilde{S}\rightarrow S$, with $\tilde{S}$ a bielliptic surface of type 1 (see Paragraph \ref{sub:biellipticover} for more details). We first investigate the properties of the induced Brauer map  for these types of covers finding how this behaves differently in the two cases:
\begin{introthm}\label{thmb}
\begin{enumerate}[leftmargin=0cm,itemindent=1cm,labelwidth=.5cm,labelsep=.1cm,align=left,label=(\alph*)]
    \item If $S$ is of type 2, then $\tilde{\pi}_{\operatorname{Br}}:\operatorname{Br}(S)\rightarrow\operatorname{Br}(\tilde{S})$ is trivial.
    \item If $S$ is of type 3, then $\tilde{\pi}_{\operatorname{Br}}:\operatorname{Br}(S)\rightarrow\operatorname{Br}(\tilde{S})$ is injective.
\end{enumerate}
\end{introthm}
The main tool behind our argument is a result of Beauville (see Section \ref{sec:oes} for more details) which states that the kernel of the Brauer map of a cyclic \'etale cover $X\rightarrow X/\sigma$ is naturally isomorphic to the kernel of the norm map $\operatorname{Nm}:\Pic(X)\rightarrow\Pic(X/\sigma)$ quotiented by $\operatorname{Im}(1-\sigma^*)$. We prove that a line bundle on $\tilde{S}$ is in the kernel of the norm map only if it is numerically trivial. Then we reach our conclusion by carefully computing the norm map of numerically trivial line bundles. The different behaviors of the two type of surfaces is motivated by the different "values" taken by the norm map on torsion elements of $H^2(\tilde{S},\Z)$: in the type 2 case they are sent to topologically trivial line bundles, while this is not true in the type 3 case.\par
Theorem \ref{thmb} is interesting in itself, and some parts of its proof will be useful in order to study the Brauer map  to the canonical cover for bielliptic surfaces of type 2.\par
We then turn our attention to the main focus of this paper: to give a complete description for the Brauer map to the canonical cover of any bielliptic surface. We first show, similarly to what happens for Enriques surfaces, that  the Brauer map is injective for a general bielliptic surface. In particular we show the following general statement:
\begin{introthm}\label{thma}
Given a bielliptic surface $S$, let $\pi:X\rightarrow S$ be its canonical cover. If the two elliptic curves $A$ and $B$ are not isogenous, then the pullback map
$$
\pi_{\operatorname{Br}}:\operatorname{Br}(S)\rightarrow\operatorname{Br}(X)$$
is injective.
\end{introthm}
When $A$ and $B$ are isogenous one encounters the first examples of bielliptic surfaces with non injective Brauer map. It turns out that the behavior of the Brauer map depends on the explicit action of the group $G$ on the product $A\times B$. This is showed in Theorems \ref{thm:injective}, \ref{thm:type1}, \ref{thm:t1trivial},  \ref{thm:type3}, \ref{thm:type5}, and \ref{thm:type 2}, where we give necessary and sufficient conditions for the Brauer map to be injective, trivial, and, in the case of type 1 surfaces (whose Brauer group is not simple), neither trivial nor injective. Our very explicit results allow us to construct examples of bielliptic surfaces exhibiting each of the possible behaviors of the Brauer map. Unfortunately the statements are  involved and it is not possible reproduce them here without a lengthy explanation of the notation used. The rough geometric picture is the following:
\begin{itemize}[leftmargin=0cm,itemindent=1cm,labelwidth=.5cm,labelsep=.1cm,align=left,]
    \item[{\bfseries Type 1:}] These  surfaces are constructed by choosing two elliptic curves $A$ and $B$ and a 2-torsion point on $A$. Thus the moduli space has dimension 2. In order to have a non injective Brauer map one can choose freely the elliptic curve $B$, but has only finitely many possibilities for $A$ and the 2-torsion point. Thus we obtain a 1- dimensional family. In Example \ref{not count} we show that uncountably many such surfaces exist. On the other hand only countably many type 1 bielliptic surfaces can have a trivial Brauer map to their canonical cover. In fact, to obtain a trivial Brauer map one has to choose the ellipic curve $B$ among those having complex multiplication.
    \item[{\bfseries Type 2:}] These  surfaces are constructed by choosing two elliptic curves $A$ and $B$ and two 2-torsion points, one on  $A$ and one on $B$. Hences the moduli space has dimension 1. Similarly to what happens in the previous case, in order to have a non injective (and hence trivial) Brauer map only the choice of the curve $B$ can be made freely, while $A$ must be taken among finitely many possibilities. We show  in this case (cf. Example \ref{last}(a)) that if the curves $A$ and $B$ are isomorphic, regardless of the choice of the torsion points, the Brauer map is trivial.
    \item[{\bfseries Type 3:}] These  surfaces are constructed by choosing one elliptic curve $A$ and a 4-torsion point on it. Therefore the moduli space has dimension 1. In order to have a non injective (and hence trivial) Brauer map, $A$ must be isogenous to the curve with $j$-invariant 1728. Thus there are only finitely many such surfaces.    \item[{\bfseries Type 5:}] These  surfaces are constructed by choosing one elliptic curve $A$ and a 3-torsion point on it. We deduce that the moduli space has dimension 1. In order to have a non injective (and hence trivial) Brauer map, $A$ must be isogenous to the curve with $j$-invariant 0. Thus, as in the previous case,  there are only finitely many such surfaces.   
\end{itemize}

The proof of Theorem \ref{thma} uses the same ideas as in the proof of Theorem \ref{thmb}. In fact we can leverage on the fact that $X$ and $S$ have the same Picard number (as for the case of a bielliptic cover) to show that line bundles in the kernel of the norm map are topologically trivial. The result is then obtained by showing that line bundles in $\Pic^0(X)$ which are also in the kernel of the norm map are always in $\operatorname{Im}(1-\sigma^*).$
As a corollary of both Theorem \ref{thmb} and \ref{thma} we find an example of isogeny between two abelian varieties $\varphi:X\rightarrow Y$ such that the corresponding group homomorphism $\varphi_{\operatorname{Br}}$ is not injective, see Section \ref{5.2}.\par

\indent This paper is organized as follows. Section \ref{sec:oes} contains all the background and preliminary results. More precisely we outline some classical facts on the geometry of bielliptic surfaces, and present the construction, due to Nuer, of the bielliptic covers of surfaces of type 2 and 3. We also expose the work of Beauville \cite{Be2009} which allows us to study the kernel of the Brauer map in terms of the norm homomorphism of the cover. We conclude the section by describing the Neron--Severi group of a product of elliptic curves. In Section \ref{sec:generators} we provide explicit generators for $H^2(S,\Z)_{\text{tor}}$, when $S$ is a bielliptic surface of type 1, 2, 3 or 5.  We prove Theorem \ref{thmb} in Section \ref{sec:biellipticcover}, while we completely describe the norm map to the canonical cover in Section \ref{sec:abelian}. Here we also construct examples of bielliptic surfaces of every type in which the Brauer map behaves differently. In the Appendix, which is joint work of the second author of the main paper with J. Bergström, a structure theorem for the homomorphism ring of two elliptic curves is given in the case of $j$-invariant 0 or 1728. This will give, in turn, a really useful description of the Picard group of the product of such curves, which is fundamental to study the Brauer map of bielliptic surfaces of type 3 and 5.

\begin{notation}
We are working over the field of complex numbers $\mathbb{C}$. If $X$ is a complex abelian variety over $\mathbb{C}$, and $n\in\Z$, then $X[n]$ will denote the subscheme of $n$-torsion points of $X$, while $n_X:X\rightarrow X$ will stand for the "multiplication by $n$ isogeny". Given $x\in X$ a point, then the translation by $x$ will be denoted as $t_x$. In addition, if $\dim X=1$ -- that is, $X$ is an elliptic curve -- then $P_x$ will be the line bundle $\sO_X(x-p_0)\simeq t_{-x}^*\sO_X(p_0)\otimes\sO_X(-p_0)$ in $\Pic^0(X)$, where $p_0\in X$ is the identity element.\par
For any smooth complex projective variety $Y$ we will denote the identity homomorphism as $1_Y$ (or simply 1 if there is no chance of confusion), while $K_Y$ and $\omega_Y$ will stand for the canonical divisor class and the dualizing sheaf on $Y$, respectively. If $D$ and $E$ are two linearly equivalent divisors on $Y$ we will write $D\sim E$; in addition $\sO_Y(D)$ will denote the line bundle associated to the divisor $D$.
\end{notation}
 
\begin{ack}
All the authors were partially supported by the grant 261756 of
the Research Councils of Norway. We are grateful to H. Nuer for sharing with us an early draft of \cite{Howie} and the construction of a bielliptic cover of bielliptic surfaces of type 2 and 3. We would also like to thank T. Suwa for explaining many facts about the torsion of the $H^2(X,\Z)$ for bielliptic surfaces.
\end{ack}

\section{Background and preliminary results}\label{sec:oes}
\subsection{Bielliptic surfaces}\label{subs:bielliptic}
A complex \emph{bielliptic} (or \emph{hyperelliptic}) surface $S$ is a minimal smooth projective surface over the field of complex numbers with Kodaira dimension $\kappa(S) = 0$, irregularity $q(S) = 1$, and geometric genus $p_g(S) = 0$. By the work of Bagenera-- De Franchis (see for example \cite[10.24-10.27]{Ba2013}), the canonical bundle $\omega_S$ has order either $2,\;3,\;4$ or 6 in $\operatorname{Pic}(S)$, and $S$ occurs as a finite étale quotient of a product $A\times B$ of elliptic curves by a finite group $G$ acting on $A$ by translations and on $B$ such that $B/G \simeq \mathbb{P}^1$. More precisely we have the following classification result.

\begin{thm}[Bagnera--De Franchis {\cite{BdF},  \cite[Theorem at p. 473]{Suwa1970OnHS}, \cite[p. 37]{BM1977}}]\label{thm:bagneradef}
A bielliptic surface is of the form $S =
A\times B/G$, where  $A$ and $B$ are elliptic curves and $G$  a finite group of translations of $A$ acting on $B$ by automorphisms. They are divided into seven types according to $G$ as shown in Table \ref{table:ta}.
\end{thm}

\noindent There are natural maps $a_S \colon S \to A/G$ and $g\colon S \to B/G \simeq \mathbb{P}^1$ which are both elliptic fibrations. The morphism $a_s$ is smooth, and coincides with the Albanese morphism of $S$. On the other hand, $g$ admits multiple fibers, corresponding to the branch points of the quotient $B\rightarrow B/G$, with multiplicity equal to that of the associated branch point.  The smooth fibers of $a_S$ and $g$ are isomorphic to $B$ and $A$, respectively.  We will denote by $a$ and $b$ the classes of these fibers in $\Num(S)$, $H^2(S, \Z)$ and $H^2(S, \Q)$. \\

It is well known (see for example \cite[p. 529]{Se1990}) that $a$ and $b$ span $H^2(S, \Q)$ and satisfy $a^2 = b^2 = 0$, $ab = |{G}|$. Furthermore, we have the following description of the second cohomology of $S$: 

\begin{prop}\label{prop:torsionSerrano}
The decomposition of $H^2(S, \Z)$ is described according to the type of $S$ and the multiplicities $(m_1, \ldots, m_s)$ of the singular fibers of $g \colon S \to \mathbb{P}^1$ as follows: 

\begin{center}
\renewcommand{\arraystretch}{1.4}
\begin{tabular}{l*{4}{c}}
\hline

Type & $(m_1, \ldots, m_s)$ & $H^2(S, \Z)$&$H^2(S, \Z)_{\text{\emph{tor}}}$ \\

\hline
$1$ & $(2,2,2,2)$ & $\Z[\frac{1}{2}a] \oplus \Z[b] \oplus \Z/2\Z \oplus \Z/2\Z$ &$\Z/2\Z\times\Z/2\Z$\\

$2$ & $(2,2,2,2)$ & $\Z[\frac{1}{2}a] \oplus \Z[\frac{1}{2}b] \oplus \Z/2\Z$ &$\Z/2\Z$\\

$3$ & $(2,4,4)$ & $\Z[\frac{1}{4}a]\oplus \Z[b] \oplus \Z/2\Z$ &$\Z/2\Z$\\

$4$ & $(2,4,4)$ & $\Z[\frac{1}{4}a] \oplus \Z[\frac{1}{2}b]$ &$0$\\

$5$ & $(3,3,3)$ & $\Z[\frac{1}{3}a]\oplus \Z[b] \oplus \Z/3\Z$ &$\Z/3\Z$\\

$6$ & $(3,3,3)$ & $\Z[\frac{1}{3}a]\oplus \Z[\frac{1}{3}b]$&$0$ \\

$7$ & $(2,3,6)$ & $\Z[\frac{1}{6}a]\oplus \Z[b]$&$0$\\
\hline
\end{tabular}
\end{center} 
\end{prop}
\qquad
\begin{proof}
See \cite[Tables 2 and 3]{Se1990}. The computation of the torsion of $H^2(S,\Z)$ can be also found in \cite{iitaka1970,Ser1991,Suwa1970OnHS,umemura1975}.
\end{proof}
\qquad

\noindent Since $H^2(S, \sO_S) = 0$, the first Chern class map $c_1\colon \Pic(S) \to H^2(S, \Z)$ is surjective, so the Néron-Severi group $NS(S) \simeq H^2(S, \Z)$. Modulo torsion we then get \[\operatorname{Num}(S) = \Z[a_0] \oplus \Z[b_0]\] where $a_0 = \frac{1}{\ord{(\omega_S})}a$ and $b_0 = \frac{\ord{(\omega_S)}}{|G|}b.$ 

\subsection{Canonical covers}\label{sub:cancover}
Let $S$ be a bielliptic surface and denote by $n$ the order of its canonical bundle. Then, by a classical construction (see for example \cite[Section 2]{BM1998}), $\omega_S$ induces an \'etale cyclic cover $\pi_S\colon X \to S$, called the \emph{canonical cover of $S$}. From now on, when there is no confusion, we will omit the subscript $S$ and write simply $\pi:X\rightarrow S$.\par
If we let $\lambda_S \coloneqq |G|/\ord{(\omega_S)}$, we have that $G\simeq \Z/n\Z \oplus \Z/\lambda_S\Z$, and $X$ is the abelian surface sitting as an intermediate quotient
$$
\xymatrix{A\times B \ar[dr]\ar[rr]&&S\simeq A\times B/G\\
&X\simeq A\times B/H\ar[ur]_\pi&}
$$
where $H \simeq \Z/\lambda_S\Z$. The abelian surface $X$ thus comes with homomorphisms of abelian varieties $p_A:X \to A/H$ and $p_B:X\to B/H$ with kernels isomorphic to $B$ and $A$, respectively. Denoting by $a_X$ and $b_X$ the classes of the fibers $A$ and $B$ in $\operatorname{Num}(X)$, we have $a_X \cdot b_X = \lambda_S$ and the embedding $\pi^* \colon \operatorname{Num}(S) \hookrightarrow{} \operatorname{Num}(X)$ satisfies 
\begin{equation}\label{eq:num abelian}\pi^*a_0 = a_X, \pi^*b_0 = \frac{n}{\lambda_S}b_X.
\end{equation}
There is a fixed-point-free action of the group $\Z/n\Z$ on the abelian variety $X$ such that the quotient is exactly $S$. We will denote by $\sigma\in \operatorname{Aut}(X)$ a generator of $\Z/n\Z$. In what follows it will be useful to have an explicit description of $\sigma$ when $S$ is of type 1, 2, 3, or 5.\par
Suppose first that $S$ is of type 1, 3, or 5, so $G$ is cyclic, $H$ is trivial, and $X\simeq A\times B$. If $S$ is of type 3, then the $j$-invariant of $B$ is 1728, and $B$ admits an automorphism $\omega:B\rightarrow B$ of order 4. If $S$ is of type 5, $B$ has $j$-invariant 0 and admits an automorphism $\rho$ of order 3 (see for example\cite[p. 37]{BM1977}, \cite[List 10.27]{Ba2013} or \cite[p. 199]{BHPVdV}). With this notation we have that the automorphism $\sigma$ of $A\times B$ inducing the covering $\pi$ is given by
\begin{equation}\label{eq:absigma}
   \sigma(x,y)=
   \begin{cases}
   (x+\tau,-y),&\quad\text{if $S$ is of type 1,}\\
   (x+\epsilon,\omega(y)),&\quad\text{if $S$ is of type 3,}\\
   (x+\eta,\rho(y)),&\quad\text{if $S$ is of type 5,}\\
   \end{cases}
  \end{equation}
where $\tau$, $\epsilon$, and $\eta$ are points of $A$ of order 2, 4, and 3 respectively. We remark that different choices for the automorphism $\rho$ and $\omega$ - there are two possible choices in each case- will lead to isomorphic bielliptic surfaces.\par
If $S$ is otherwise of type 2, then there are points $\theta_1\in A$ and $\theta_2\in B$, both of order two, such that $X$ is the quotient of $A\times B$ by the involution $(x,y)\mapsto(x+\theta_1,y+\theta_2)$. If we denote by $[x,y]$ the image of $(x,y)$ through the quotient map, we have that 
\begin{equation}\label{eq:absigmatype2}
   \sigma[x,y]=[x+\tau,-y],
  \end{equation}
where $\tau\in A$ is a point of order 2, $\tau\neq \theta_1$.
\subsection{Covers of bielliptic surfaces by other bielliptic surfaces}\label{sub:biellipticover} When $G$ is not a cyclic group, or when $G$ is cyclic, but the order of $G$ is not a prime number, then the bielliptic surface $S$ admits a cyclic cover $\tilde{\pi}:\tilde{S}\rightarrow S$, where $\tilde{S}$ is another bielliptic surface. This construction, together with the statement of Lemma \ref{lem:type3}, appears explicitly in the unpublished work of Nuer \cite{Howie},and is implicit in the work of Suwa \cite[p. 475]{Suwa1970OnHS}. 
The main point that we will need in Section \ref{sec:biellipticcover} is the description of the pull-back map $\operatorname{Num}(S)\rightarrow\operatorname{Num}(\tilde{S})$.\par

\begin{lemma}\label{lem:type3}\label{lem:type2}
\begin{enumerate}[leftmargin=0cm,itemindent=1cm,labelwidth=.5cm,labelsep=.1cm,align=left,label=(\roman*)]
\item Let $S$ be a bielliptic surface such that $\operatorname{ord}(\omega_S)$ is not a prime number and take  $d$  a proper divisor of $n$. Then there is a bielliptic surface $\tilde{S}$ sitting as an intermediate étale cover between $S$ and $X$, 
$$\xymatrix{X\ar@/_1pc/[rrrr]_{\pi_S}\ar[rr]^{\pi_{\tilde{S}}}&&\tilde{S}\ar[rr]^{\tilde{\pi}} &&S}$$
such that $\ord{(\omega_{\tilde{S}})} = \frac{\ord{(\omega_S})}{d}$ and \[\tilde{\pi}^*a_0 = \tilde{a_0}, \tilde{\pi}^*b_0 = d\tilde{b_0},\] where $\tilde{a_0}, \tilde{b_0}$ are the natural generators of $\operatorname{Num}{(\tilde{S})}$. 
\item Let $S$ be a bielliptic surface with $\lambda_S > 1$, i.e., with $G$ not cyclic. Then there is a bielliptic surface $\tilde{S}$ sitting as an intermediate étale cover between $S$ and $A\times B$ $$\xymatrix{A\times B\ar@/_1pc/[rrrr]_{\pi_S}\ar[rr]^{\pi_{\tilde{S}}}&&\tilde{S}\ar[rr]^{\tilde{\pi}} &&S}$$
such that $\lambda_{\tilde{S}} = 1$, $\operatorname{ord}(\omega_{\tilde{S}}) = \operatorname{ord}(\omega_{S})$ and 
\[\tilde{\pi}^*a_0 = \lambda_{S}\tilde{a_0}, \tilde{\pi}^*b_0 = \tilde{b_0},\] 
where $\tilde{a_0}, \tilde{b_0}$ are the natural generators of $\operatorname{Num}{(\tilde{S})}$. 
 \end{enumerate}
\end{lemma}
In what follows we will need a more explicit construction of $\tilde{S}$, when $S$ is either of type 2 or 3.
\begin{ex}\label{ex:type3->type1}\label{ex:type2->type1}
\begin{enumerate}[leftmargin=0cm,itemindent=1cm,labelwidth=.5cm,labelsep=.1cm,align=left,label=(\alph*)]
\item Suppose that $S$ is a bielliptic surface of type 3. Then the canonical bundle has order 4. In addition the canonical cover $X$ of $S$ is a product of elliptic curves, that is $X\simeq A\times B$. Using the notation of \eqref{eq:absigma}, we obtain $\tilde{S}$ from $A\times B$ by taking the quotient with respect to the involution $(x,y)\mapsto (x+2\epsilon,-b)$. Thus we have that $\tilde{S}$ is a bielliptic surface of type 1. The map $\tilde{\pi}:\tilde{S}\rightarrow S$ is an \'etale double cover with associated involution $\tilde{\sigma}$. Hence, given $s\in\tilde{S}$, we can see it as an equivalence class $[x,y]$ of a point $(x,y)\in A\times B$. Then we have an explicit expression for $\tilde{\sigma}$:
\begin{equation}\label{eq:explicitsigma3}
    \tilde{\sigma}(s)=[x+\epsilon,\omega(y)].
    \end{equation}
    \item Suppose that $S$ is a bielliptic surface of type 2, so the group $G$ is isomorphic to the product $\Z/2\Z\times\Z/2\Z$. Then we obtain $\tilde{S} $ from $A\times B$ by taking the quotient with respect to $(x,y)\mapsto (x+\tau,-y)$, where we are using the notation of \eqref{eq:absigmatype2}. Thus, as in \ref{ex:type3->type1}, $\tilde{S}$ is a bielliptic surface of type 1 and each $s\in\tilde{S}$ can be written as an equivalence class $[x,y]$ of a point $(x,y)\in A\times B$. If we denote again by $\tilde{\sigma}$ the involution induced by the cover $\tilde{\pi}:\tilde{S}\rightarrow S$, we have the following:
\begin{equation}\label{eq:explicitsigma2}
    \tilde{\sigma}(s)=    [x+\theta_1,y+\theta_2].
    \end{equation}
    \end{enumerate}
\end{ex}

\subsection{Norm homomorphisms} Let $\pi:X\rightarrow Y$ be a finite locally free morphism of projective varieties of degree $n$. To it we can associate a group homomorphism $\operatorname{Nm}_\pi:\Pic(X)\rightarrow\Pic(Y)$ called the \emph{norm homomorphism associated to $\pi$}. This is constructed in the following manner. First, one lets $\mathscr{B}:=\pi_*\sO_X$, and defines a  morphism of sheaves of multiplicative monoids $N:\mathscr{B}\rightarrow \sO_Y$: given $s$ a section of $\mathscr{B}$ on an open set $U$, let $m_s$ be the endomorphism of $\mathscr{B}(U)$ induced by the multiplication by $s$; we set $N(s):=\det(m_s)\in\sO_Y(U)$ (see \cite[\S\:6.4, and \S 6.5]{EGAII} or \cite[Lemma 0BD2]{stacks-project}
). The restriction of $N$ to invertible sections induces a morphism of sheaves of groups $N:\mathscr{B}^*\rightarrow\sO_Y^*$. Now, given $L$ an invertible sheaf on $X$, $\pi_*L$ is an invertible $\mathscr{B}$-module and, as such is represented by a cocycle $\{u_{ij},U_i\}$ for an open cover $\{U_i\}$ of $Y$. Observe that $u_{ij}\in\mathscr{B}^*(U_{ij})$. The fact that $N$ is multiplicative ensures that also the $v_{ij}:=N(u_{ij})$ satisfies the cocycle condition and so uniquely identifies a line bundle $\operatorname{Nm}_\pi(L)$ on $Y$. The map $L\mapsto \operatorname{Nm}_\pi(L)$ is a group homomorphism by \cite[(6.5.2.1)]{EGAII}. In addition \cite[(6.5.2.4)]{EGAII} ensures that
\begin{equation}\label{eq:normofpullback}
    \operatorname{Nm}_\pi(\pi^*M)\simeq M^{\otimes n},
\end{equation}
and we also have the following important property:
\begin{prop}
\label{prop:functorial}
Given two finite locally free morphism $\pi_1:X\rightarrow Y$ and $\pi_2:Y\rightarrow Z$, then
$$
\operatorname{Nm}_{\pi_2\circ \pi_1}=\operatorname{Nm}_{\pi_2}\circ \operatorname{Nm}_{\pi_1}
$$
\end{prop}
\begin{proof}
See \cite[Lemma 21.5.7.2]{EGAIViv}.
\end{proof}
Suppose now that $\pi:X\rightarrow Y$ is an \'etale cyclic cover of degree $n$. Then there is a fixed-point-free automorphism $\sigma:X\rightarrow X$ of order $n$ such that $Y\simeq X/\sigma$. In addition we can write $\mathscr{B}\simeq\bigoplus_{h=0}^{n-1}M^{\otimes h}$ with $M$ a line bundle of order $n$ in $\Pic(Y)$. In this particular setting the norm homomorphism satisfies some additional useful properties. First, as $\operatorname{Nm}_\pi$ behaves well with base change (\cite[Proposition 6.5.8]{EGAII}), it is not difficult to see that
\begin{equation}\label{eq:1-sigma}
\operatorname{Nm}_\pi\circ(1_X-\sigma^*)=0.
\end{equation}

In additon, as discussed by Beauville in \cite{Be2009}, we have that 
\begin{equation}\label{eq:pullbackofnorm}
\pi^*\operatorname{Nm}_\pi(L)\simeq \bigotimes_{h=0}^n(\sigma^h)^*L
\end{equation}
In fact, by the definiton of pushforward of divisors (\cite[Definition 21.5.5]{EGAIViv}), if $L\simeq\sO_X(\sum a_i\cdot D_i)$ with prime divisors on $X$, then $\operatorname{Nm}_\pi(L)\simeq\sO_Y(\sum a_i\cdot \pi_*D_i)$. Therefore \eqref{eq:pullbackofnorm} follows form the fact that for a prime divisor $D$ we have that $\pi^*\pi_*D\sim\sum_{h=0}^{n-1}(\sigma^h)^*D$.
  \begin{rmk}[$\Pic^0$ trick]\label{rmk:trick}
  In what follows it will be important to provide elements in the kernel of the Norm homomorphism. We will often use the following trick. Let $\pi:X\rightarrow Y$ be an \'etale morphism of degree $n$ and suppose that there is a line bundle $L$ on $X$ such that $\operatorname{Nm}_\pi(L)\in\Pic^0(Y).$ Then there is an element  $\alpha\in\Pic^0(X)$ such that $\operatorname{Nm}_\pi(L\otimes\alpha)$ is trivial. In fact, as abelian varieties are divisible groups, it is possible to find $\beta\in\Pic^0(Y)$ such that $\beta^{\otimes n}\simeq \operatorname{Nm}_\pi(L)^{-1}$. Then, by \eqref{eq:normofpullback} we get
  $$
  \operatorname{Nm}_\pi(L\otimes\pi^*\beta)\simeq \operatorname{Nm}_\pi(L)\otimes \beta^{\otimes n}\simeq \sO_Y.
  $$
   \end{rmk}We conclude this paragraph by saying that, from now on, if there is no possibility of confusion, we will omit the subscript when denoting the norm. That is we will write $\operatorname{Nm}$ instead of $\operatorname{Nm}_\pi$
\subsection{Brauer groups and Brauer maps}
For a scheme $X$, the \emph{cohomological Brauer group} $\operatorname{Br}'(X)$ is defined as the torsion part of the étale cohomology group $H^2_{\mathrm{et}}(X, \sO_X^*)$. For complex varieties, this is isomorphic to the torsion of $H^2(X, \sO_X^*)$ in the analytic topology. In addition, when $X$ is quasi-compact and separated, by a theorem of Gabber (see,  for example, \cite{dejongGabber} for  more details) the cohomological Brauer group of $X$ is canonically isomorphic to the \emph{Brauer group} $\operatorname{Br}(X)$ of Morita-equivalence classes of Azumaya algebras on $X$. For what it concerns the present document, we will only be concerned with smooth complex projective varieties, therefore all these three groups will be isomorphic and will be denoted simply by $\operatorname{Br}(X)$. Furthermore we will only speak of the \emph{Brauer group of $X$}, without any additional connotation.\par
If $S$ is a bielliptic surface, the exponential sequence yields that $H^3(S, \Z) \simeq H^2(S, \sO_S^*)$, so that the Brauer group of $S$ is isomorphic to the torsion of $H^3(S, \Z)$. By Poincar\'e duality and the universal coefficients theorem, the torsion of $H^3(S, \Z)$ is (non canonically) isomorphic to the torsion of $H^2(S, \Z)$, so the isomorphism type of the Brauer group of $S$ can be deduced in terms of Proposition 2.2. \\
Crucial to our purposes will be the following result of Beauville which describes the kernel of the Brauer map $\pi_{\operatorname{Br}}$ when $\pi$ is a cyclic \'etale cover.
\begin{prop}[{\cite[Prop. 4.1]{Be2009}}]\label{prop:Beauville}
Let $\pi\colon X\to S$ be an étale cyclic covering of smooth projective varieties. Let $\sigma$ be a generator of the Galois group of $\pi$, $\operatorname{Nm}\colon \Pic(X) \to \Pic(S)$ be the norm map and $\pi_{\operatorname{Br}} \colon \operatorname{Br}(S) \to \operatorname{Br}(X)$ be the pullback. Then we have a canonical isomorphism \[\operatorname{Ker}(\pi_{\operatorname{Br}} ) \simeq \operatorname{Ker}\mathrm{Nm}/(1-\sigma^*)\Pic(X).\] 
\end{prop}

\subsection{The Neron--Severi of a product of elliptic curves}\label{subsec:NUMprod}
In this paragraph we want to describe $\operatorname{Num}(A\times B)$ when $A$ and $B$ are two elliptic curves. We will do so by using the identification of $\operatorname{Num}(X)\simeq NS(X)$ which holds for abelian surfaces. We believe that many of these topics might be well known by experts, but we were not able to find a rigorous literature, thus we wrote this for the reader convenience. In the first part of this paragraph we will follow closely the narrative of \cite{KLT2019}.\par
Let $A$ be an elliptic curve over $\mathbb{C}$ with identity element $p_0$, then there is a lattice $\Lambda$  such that $A\simeq \mathbb{C}/\Lambda$. Identify $A$ with its dual and consider $\sP_A$ the \emph{normalized Poincar\'e bundle} on $A\times A$:
$$
\sP_A\; \:\simeq\; \sO_{A\times A}(\Delta_A)\, \otimes \, {\rm pr}_1^*\sO_A(-p_0)\, \otimes \, {\rm pr}_2^*\sO_A(-p_0)
$$
where $\Delta_A \subset A\times A$ is the diagonal divisor and ${\rm pr}_1$, ${\rm pr}_2$ are the projections of $A\times A$ onto 
the first and second factor respectively. Observe that if $x$ is a point in $A$, then the topologically trivial line bundle $P_x$ is simply ${\sP_A}_{|A\times\set{x}}\simeq{\sP_A}_{|\set{x}\times A}$.\par

Given another elliptic curve $B$, line bundles $L_A$ and $L_B$ on $A$ and $B$ respectively,
and a morphism $\varphi:B\rightarrow A$, we define a line bundle on the product $A\times B$
\begin{equation}\label{Picard}
  L(L_A, L_B, \varphi) \, := \, (1_A\times \varphi)^*\sP_A \otimes {\rm pr}_A^*L_A\otimes {\rm pr}_B^*L_B
 \end{equation}
 where ${\rm pr}_A$ and ${\rm pr}_B$ are the projections onto $A$ and $B$ respectively. As a direct consequence of the see-saw principle it is possible to see that, if $M_A$ and $M_B$ are two other line bundles on $A$ and $B$, and $\psi:B\rightarrow A$ is another homomorphism, then
 $$L(L_A\otimes M_A, L_B\otimes M_B,\varphi+\psi)\, \simeq \, L(L_A, L_B,\varphi) \otimes L(M_A,M_B,\psi).$$
In addition, the universal property of the dual abelian variety ensures that every line bundle $L\in\Pic(A\times B)$ is of the form $L(L_A,L_B,\varphi)$ for some invertible sheaves $L_A$ and $L_B$ and a morphism $\varphi$. Therefore  we have an isomorphism
$$\Pic(A\times B)\simeq\Pic(A)\times\Pic(B)\times\operatorname{Hom}(B,A).$$
If we quotient by numerically trivial line bundles, we find that
\begin{equation}\label{eq:h2}
H^2(A\times B,\Z)\simeq\operatorname{Num}(A\times B)\simeq \Z\cdot [B]\times \Z\cdot[A]\times \operatorname{Hom}(B,A),
\end{equation}
where $[A]$ and $[B]$ are the classes of the fibers of the two projections. Let us denote by $l(\deg(L_A),\deg(L_B),\varphi)$ the first Chern class of $L(L_A,L_B,\varphi)$. Then every class in $\operatorname{Num}(A\times B)$ can be written as $l(m,n,\varphi)$ for some integers $n$ and $m$ and an isogeny $\varphi$. In what follows we will often refer to line bundles (or numerical classes) in $\operatorname{Hom}(B,A)$ as elements of the \emph{Hom-part} of $\operatorname{Pic}(A\times B)$ (or of $\operatorname{Num}(A\times B))$. For our purposes it will be really important to pick explicit generators for $\operatorname{Num}(A\times B)$ to see how the automorphism $\sigma$ acts on $H^2(A\times B,\mathbb{Z})$. In order to do that, we need to investigate the $\mathbb{Z}$-module structure on $\operatorname{Hom}(B,A)$. \par
So suppose that there is a nontrivial isogeny $\varphi:B\rightarrow A$. Then we know that $\operatorname{Hom}(B,A)$ has rank 1 if $A$ does not have complex multiplication, and 2 otherwise (more details about elliptic curves with complex multiplication can be found in the Appendix).\par
Suppose the first, so that there exists an isogeny $\psi:B\rightarrow A$ such that $l(0,0,\psi)$ generates the Hom-part of $H^2(A\times B, \Z)$. We will call such isogeny a \emph{generating isogeny for $\operatorname{Num}(A\times B)$}. Observe that, since $l(0,0,\psi)$ is necessarily a primitive class, $\psi$ cannot factor through any "multiplication by n" map. That is, we cannot write $\psi=n\cdot\psi'$ for any $n$. In particular, for any integer $n$ we have that $\operatorname{Ker}\psi$ does not contain $B[n]$ as a subscheme.\par
Suppose now that $A$ has complex multiplication, and again fix a non trivial isogeny $\varphi:B\rightarrow A$. Then also $B$ has complex multiplication, and $\operatorname{Hom}(B,A)$ is a rank 2 free $\mathbb{Z}$-module. We pick generators $\psi_1$ and $\psi_2$, and we have that for any line bundle $L$ on $A\times B$ there are two integers $h$ and $k$ such that
\begin{equation}\label{eq:linebundleAB}
L\simeq L(M_A,M_B,h\cdot\psi_1+k\cdot\psi_2),
\end{equation}
where $M_A$ and $M_B$ are element of $\Pic(A)$ and $\Pic(B)$ respectively. In addition we can write 
\begin{equation}\label{eq:numAB}
H^2(A\times B,\Z)=\left<l(1,0,0),l(0,1,0), l(0,0,\psi_1),l(0,0,\psi_2)\right>.
\end{equation}
In the particular cases in which the $j$-invariant of $B$ is either 0 or 1728, then Theorem \ref{thm:appendix2} in the Appendix yields a more accurate description. In fact, if we denote by $\lambda_B:B\rightarrow B$ the automorphism $\rho$ or $\omega$ (see again the Appendix or Paragraph \ref{sub:cancover}), we have that there exist an isogeny $\psi:B\rightarrow A$ such that, in \eqref{eq:linebundleAB} and \eqref{eq:numAB} we can take $\psi_1=\psi$ and $\psi_2=\psi\circ\lambda_B$. So we have that
\begin{equation}\label{eq:numABII}
H^2(A\times B,\Z)=\left<l(1,0,0),l(0,1,0), l(0,0,\psi),l(0,0,\psi\circ\lambda_B)\right>.
\end{equation}
In this case we say that $\psi$ is again a \emph{generating isogeny} for $H^2(A\times B,\Z)$.
Observe again the isogenies $\psi_i$, as well as $\psi$, cannot factor through the multiplication by an integer or they could not generate the whole $\Hom(B,A).$\par


\section{Generators for the torsion of the second cohomology for bielliptic surfaces}\label{sec:generators}

In this section we give explicit generators for the torsion of $H^2(S, \Z)$ in terms of the reduced multiple fibers of the elliptic fibration $g\colon S \to \mathbb{P}^1$.  More precisely we will prove the following statement:
\begin{prop}\label{prop:generators}
Let $S = A\times B /G$ be a bielliptic surface. Denote by $D_i$ the reduced multiple fibers of $g \colon S \to \mathbb{P}^1$ with the same multiplicity. Then the torsion of $H^2(S, \Z)$ is generated by the classes of differences $D_i - D_j$ for $i \neq j$. 
\end{prop}
The reader who is familiar with the work of Serrano might find similarities between the above statement and Serrano's description of the torsion of $H^2(X,\Z)$ when there is an elliptic fibration $\varphi:X\rightarrow C$ with multiple fibers (cfr. \cite[Corollary 1.5 and Proposition 1.6]{Ser1990}). However in \cite{Ser1990} it is used the additional assumption that $h^1(X,\sO_X)=h^1(C,\sO_C)$. This clearly does not hold in our context.\\

Before proving Proposition \ref{prop:generators} we need two preliminary Lemmas. 

\begin{lemma} Let $g\colon S\to \mathbb{P}^1$ be an elliptic pencil with connected fibers. Let $D_1$ and $D_2$ be two reduced multiple fibers. Let $m_1$ and $m_2$ be the corresponding multiplicities. Then, for all non negative integers $n$,
\begin{equation}
D_1\nsim nD_2.
\end{equation} 

\end{lemma}
\begin{proof}
The statement is obvious for $n=0$, so one has to check for $n>0$. By contradiction, assume $D_1\sim nD_2$, and let $F$ be the generic fiber of $g$. Then
\begin{align}
h^0(S,\sO_S(F))&=h^0(\mathbb{P}^1,g_{\ast}\sO_S(F)) \nonumber\\
&=h^0(\mathbb{P}^1,\sO_{\mathbb{P}^1}(1)\otimes g_{\ast}\sO_S) \nonumber\\
&=h^0(\mathbb{P}^1,\sO_{\mathbb{P}^1}(1))=2. \nonumber
\end{align}
Since $h^0(S,\sO_S(D_1))\leq h^0(S,\sO_S(m_1D_1))=h^0(S,\sO_S(F))$, it follows that $h^0(S,\sO_S(D_1))\leq 2$.\\
The absurd hypothesis is used here: if $D_1\sim nD_2$, then, since the supports of $D_1$ and $D_2$ are disjoint, $H^0(S,\sO_S(D_1))$ has at least two independent sections, and therefore the dimension of $H^0(S,\sO_S(D_1))$ is 2. Thus, since $D_1^2=0$ implies that there are no basepoints (see for example \cite[II.5]{Be1996}), the map is actually a morphism $\varphi_{|D_1|}:\,S\longrightarrow \mathbb{P}^1$. Note that both $D_1$ and $nD_2$ are fibers of this morphism.\par
Let now $C$ be the generic fiber of $\varphi$ (which is irreducible by semicontinuity). Since $C\cdot D_1=0$, one gets $C\cdot F=0$ for any fiber $F$ of $g$. This implies that $g$ and $\varphi_{|D_1|}$ have the same generic fiber. So one can write $C=F$ for a fiber $F$ of $g$. But then
\begin{equation*}
D_1\sim F\sim m_1D_1,
\end{equation*}
which in turn implies that $\sO_S(D_1)^{\otimes(m_1-1)}\simeq \sO_S$, which is a contradiction.
\end{proof}

\begin{lemma}\label{lem:nontrivial}
Let $S = A\times B/G$ be a bielliptic surface with its fibrations $f\colon S \to A/G$ and $g\colon S \to \mathbb{P}^1$. Let $D_1$ and $D_2$ be two reduced multiple fibers of $g$. Then the restriction of $\sO_S(D_1 - D_2)$ to the generic fiber of $a_S$ is trivial.
\end{lemma}

\begin{proof}
Let $F = g^{-1}(p)$ be a smooth fiber of $g$. Here $p$ is the orbit $G\cdot y$ of a point $y\in B$ not fixed under any element of $G$. We will choose an embedding of $A$ into $S$ via an isomorphism $\varphi \colon A \to F$ such that we get a commutative diagram 

\begin{equation*}
\begin{tikzcd}[column sep = tiny]
A \arrow[d, "\varphi"] \arrow[rr, hook, "j"] && A \times B \arrow[d, "\pi"] \\
F \arrow[rr, hook, "i"] \arrow{dr}[swap]{\psi} && S\arrow{dl}{a_S} \\
& A/G & 
\end{tikzcd}
\end{equation*}
where $i$ is just the natural inclusion of the fiber $F$ into $S$ and $\pi$ is the quotient map. To this end we let $\varphi \colon A \to F$ be the isomorphism $a\mapsto G\cdot(x,y)$ and $j$ be the embedding $a\mapsto (x,y)$.  
Since the multiple fibers $D_i$ are images of $A\times \{y_i\}$, $i=1,2$, where the $y_i \in B$ are points fixed under a subgroup of $G$ of order equal to the multiplicity of $D_i$, we have that $\pi^* \sO_S(D_1 - D_2) = p^*\sO_B(y_1 - y_2)$ where $p_B$ is the projection $A\times B \to B$ and $y_1, y_2 \in B$ are the points corresponding to $D_1$, $D_2$, respectively. Then 
\begin{align*}
\varphi^* i^* \sO_S(D_1-D_2) &\simeq j^* \pi^* \sO_S(D_1-D_2)\\
&\simeq j^*p_B^*\sO_B(y_1 - y_2)
\end{align*}
As $p\circ j$ is the constant map we have that this is clearly trivial. Hence $\varphi^*i^*\sO_S(D_1 - D_2)$ is trivial, and since $\varphi$ is an isomorphism we deduce the statement.
\end{proof}

For the remainder, we identify $F$ and $A$ via the isomorphism $\varphi$ defined in the proof above. So we get the following commutative triangle.

\begin{equation}\label{eq:diagram2}
\begin{tikzcd}[column sep = tiny]
A \arrow[rr, hook, "i"] \arrow{dr}[swap]{\psi} && S \arrow{dl}{a_S} \\
& A/G &
\end{tikzcd}
\end{equation} 
Note that $\psi$ is an isogeny of degree $|G|$. In particular we have also that the dual isogeny $\psi^*:\Pic^0(S)\rightarrow\Pic^0(A)$ has degree $|G|$ (see, for example \cite[Proposition 2.4.3]{BL2013}).\par
With these observations we are now ready to start proving Proposition \ref{prop:generators}.  We first remark that by the canonical bundle formula for elliptic fibrations (see e.g. \cite[Thm. 7.15]{Ba2013}) applied to $g:S \to \mathbb{P}^1$ we can write \[\omega_S \simeq g^*\sO_{\mathbb{P}^1}(-2) \otimes \sO_S\big(\sum_k (m_k - 1)D_k\big)\]
where the $D_k$ are the multiple fibers of $g$ of multiplicity $m_k$. Choosing points $p, q$ on $\mathbb{P}^1$ giving rise to the fibers $m_iD_i$ and $m_jD_j$ we get that
\begin{equation}\label{eq:canbung}
    K_S \sim -D_i - D_j +  \sum_{k\neq i,j}(m_k - 1)D_k.
\end{equation}
Since $\omega_S$ is a nontrivial element in $\Pic^0(S)$ we conclude that the classes of $-D_i - D_j$ and $\sum_{k\neq i,j}(m_k - 1)D_k$ coincide in $H^2(S, \Z)$. Moreover, we observe that  $K_S$ restricts trivially to $A$, so $\omega_S$ yields a nontrivial element in $\operatorname{Ker}{\psi^*}$. Note that if $D_i$ and $D_j$ have the same multiplicity $m$, the difference $D_i - D_j$ induces a (possibly trivial) torsion element in $H^2(S, \Z)$ of order $m$. We prove Proposition \ref{prop:generators} by showing that a sufficient number of these is nontrivial so to generate the torsion of $H^2(S,\Z)$. We proceed by a case by case analysis, studying separately bielliptic surfaces of type 1, 2, 3, and 5. The key point in the argument is the observation that, if $[D_i-D_j]$ is trivial, then the line bundle $\sO_S(D_i-D_j)$ belongs to $\Pic^0(S)$. In addition, using Lemma 3.2 and the diagram \eqref{eq:diagram2}, we would have that $\psi^*\sO_S(D_i-D_j)\simeq \sO_S$, in particular $\sO_S(D_i-D_j) \in \operatorname{Ker}\psi^*$, while Lemma \ref{lem:nontrivial} ensures that $\sO_S(D_i-D_j)$ cannot be $\sO_S$. A closer study of the structure of $\Ker\psi^*\simeq\hat{G}$ will bring us to the desired conclusion.
\subsection{Type 1 bielliptic surfaces} In this case we have that $\operatorname{Ker}\psi^*$ is the reduced group scheme $\Z/2\Z$ and the fibration $g:S\rightarrow \mathbb{P}^1$ has four multiple fibers all of multiplicity 2. Hence, up to reordering the indices \eqref{eq:canbung} yields  
\begin{equation}\label{eq:canbundletype1}K_S \sim D_i - D_j + D_k - D_l.\end{equation}
In particular, as the canonical divisor is algebraically equivalent to 0, for distinct indices $i$, $j$, $k$, and $l$  we have that $D_j-D_i$ is algebraically equivalent to $D_k-D_l$. Thus we get three classes in $H^2(S,\Z)$
\begin{align}
\begin{split}\label{eq:torsiontype1}
  [D_1 - D_2] &= \{D_1 - D_2,\; D_3-D_4\}, \\
 [D_1 - D_3] &= \{D_1 - D_3,\; D_2-D_4\}, \\
 [D_1 - D_4] &= \{D_1 - D_4,\; D_2-D_3\}, \\
\end{split}
\end{align}
which a priori are neither distinct nor nontrivial. Since $H^2(S, \Z)_{\mathrm{tors}}$ is isomorphic to the Klein 4-group, we need to show that they are indeed different classes and are not zero. Note that, if two classes are equal, since they both are 2-torsion and the third classes is clearly equal to the sum of the first two, then the remaining class would be trivial. Thus it will be enough to show that for any two distinct indices the divisor $D_i-D_j$ is not algebraically equivalent to 0. Suppose otherwise that for some indices we have that $\sO_S(D_i-D_j)\in\Pic^0(S)$, then \eqref{eq:canbundletype1} would imply that also $\sO_S(D_k-D_l)$ would be in $\Pic^0(S)$. The above discussion yields that both  $\sO_S(D_i-D_j)$ and $\sO_S(D_i-D_j)$ are nontrivial elements of  $\operatorname{Ker}\psi^*$, which has only one nontrivial element, $\omega_S$. Then we can write
\[\omega_S\simeq \sO_S(D_i-D_j)\otimes\sO_S(D_k-D_l)\simeq\omega_S^{\otimes 2}\simeq\sO_S,\]
which brings a contradiction, and thus we may conclude.

\subsection{Type 2 bielliptic surfaces} 
Here $H^2(S, \Z)_{\mathrm{tors}} \simeq \Z/2\Z$, $\Ker(\psi^*) \simeq \Z/2\Z\times \Z/2\Z$ and like in the previous case there are four multiple fibers, each of multiplicity 2. As above we get the three classes induced by $D_1 - D_2$, $D_1 - D_3$ and $D_1 - D_4$, and we want to show that they cannot be all trivial. Suppose that two of these classes, say $[D_1-D_2]$ and $[D_1-D_3]$, are trivial in $H^2(S, \Z)$.  For $i=2,3$ set $L_i:=\sO_S(D_1-D_i)$ and $M_i:=\sO_S(D_i-D_4)$, then the $L_i$'s and the $M_i$'s determine nontrivial elements of $\operatorname{Ker}\psi^*$, which has only three nonzero elements. We deduce that some of these must be the same line bundle. The only option which will not contradict Lemma \ref{lem:nontrivial} would be that $L_i\simeq M_j$ for som $i\neq j$. But then we would have  
\[\omega_S \simeq L_i\otimes M_j\simeq L_i^{\otimes 2}\simeq \sO_S,\] 
which would be a contradiction. Hence at most one of the three classes can be trivial, and indeed one is actually trivial because the two nontrivial classes must coincide, implying the third is trivial. 

\subsection{Type 3 bielliptic surfaces} 
Here $H^2(S, \Z)_{\mathrm{tors}} \simeq \Z/2\Z$ and $\Ker(\psi^*) \simeq \Z/4\Z$, but now we have two fibers of multiplicity 4 and one of multiplicity 2. Denote by $E$ the reduced multiple fiber of multiplicity 2 and by $D_1$, $D_2$ the reduced multiple fibers of multiplicity 4. By the canonical bundle formula, we get $$K_S \sim E - D_1 - D_2.$$ Then in $H^2(S,\Z)$ we have the following equalities
$$[E - 2D_1 ]= [D_2 - D_1],\quad\text{and}\quad [E-2D_2] = [D_1 - D_2].$$  
We need to show that they are not both trivial. Suppose by contradiction they are both zero in $H^2(S, \Z)$, then, as before we have that $\sO_S(E - 2D_1 )$ and $\sO_S(E-2D_2)$ are non trivial elements of $\operatorname{Ker}\psi^*$. Since  both these line bundles have order two in $\Pic(S)$, and $\operatorname{Ker}\psi^*$ has only one element of order 2, we deduce that
$$
\sO_S(E - 2D_1 )\simeq\sO_S(E-2D_2).
$$ But then 
\[\omega_S^{\otimes 2}\simeq \sO_S(E - D_1 - D_2)^{\otimes 2} \simeq \sO_S(E-2D_1)\otimes\sO_S(E-2D_2)\simeq \sO_S(E-2D_1)^{\otimes 2}\simeq\sO_S\] which is impossible because $\omega_S$ is of order 4. Therefore $E-2D_1$ and $E-2D_2$ induce the same nontrivial torsion element of $H^2(S, \Z)$. \\

\subsection{Type 5 bielliptic surfaces} 
Here $H^2(S, \Z)_{\mathrm{tors}} \simeq \Z/3\Z$, $\Ker(\psi^*) \simeq \Z/3\Z$ and there are three multiple fibers, each of multiplicity 3. By the canonical bundle formula, we get \[K_S \sim -D_i - D_j + 2D_k = (D_k - D_i) + (D_k - D_j).\] Again, $K_S$ is algebraically equivalent to zero, so we get that  $[D_k - D_i ]= [D_j - D_k]$ in $H^2(S, \Z)$. Running through the indices we get the two classes 
\begin{align*}
\begin{split}
  [D_1 - D_2] &= \{D_1 - D_2,\; D_3-D_1,\; D_2-D_3\}, \\
 [D_1 - D_3] &= \{D_1 - D_3,\; D_3-D_2,\; D_2-D_1\}.
\end{split}
\end{align*}
We need to show that they are distinct and both nontrivial. Observe that if they were the same class then both classes would be trivial, so it is enough to show that they are not the zero class. Again suppose by contradiction that  $[D_k-D_i] = 0$ in $H^2(S, \Z)$, then we can write 
\[\omega_S \simeq \sO_S(D_1-D_2)\otimes\sO_S(D_1-D_3),\] 
with $\sO_S(D_1-D_2)$ and $\sO_S(D_1-D_3)$ for nontrivial elements in $\Ker(\psi^*)$.  Neither $\sO_S(D_1-D_2)$ nor  $\sO_S(D_1-D_3)$ can be isomorphic to the canonical bundle $\omega_S$, or we would have $\sO_S(D_k-D_i)\simeq\sO_S$, contradicting Lemma \ref{lem:nontrivial}. As $\Ker\psi^*$ has only two nontrivial elements, we necessarily have
$$
\sO_S(D_1-D_2)\simeq\sO_S(D_1-D_3)
$$
and so $\sO_S(D_2-D_3)\simeq\sO_S$, which contradicts again Lemma \ref{lem:nontrivial}, thus we can conclude.
\section{The Brauer map to another bielliptic surface}\label{sec:biellipticcover}
Let $S$ be a bielliptic surface of type 2 or 3. Then by Examples \ref{ex:type3->type1} and \ref{ex:type2->type1} there is a 2:1 cyclic cover $\tilde{\pi}:\tilde{S}\rightarrow S$, where $\tilde{S}$ is a bielliptic surface of type 1. As in paragraph \ref{sub:biellipticover} , we will denote by $\tilde{\sigma}$ the involution induced by $\tilde{\pi}$. In this section we are concerned with studying the Brauer map $\tilde{\pi}_{\operatorname{Br}}:\operatorname{Br}(S)\rightarrow \operatorname{Br}(\tilde{S})$. Surpisingly we reach two antipodal conclusions, depending on the type of the bielliptic surface in object.\par
Recall that, as $\tilde{S}$ is a bielliptic surface of type 1, the elliptic fibration $q_B:\tilde{S}\rightarrow\mathbb{P}^1$ has  four multiple fibers $D_1,\ldots,D_4$ of multiplicity 2, corresponding to the four 2-torsion points of $B$. We will denote by $\tau_{ij}$ the line bundle $\sO_{\tilde{S}}(D_i-D_j)$.
\subsection{Bielliptic surfaces of type 2}
Suppose that $S$ is of type 2, and note that the involution $\tilde{\sigma}$ acts on the set of the $D_i$'s by exchanging them pairwise. Up to relabeling we can assume that $\tilde\sigma^*D_1\sim D_2$ and $\tilde\sigma^*D_3\sim D_4$. By \eqref{eq:pullbackofnorm}, we therefore have that
\begin{equation}\label{eq:tau13}
\tilde{\pi}^*(\operatorname{Nm}(\tau_{13}))\simeq \tau_{13}\otimes\sigma^*\tau_{13}\simeq\tau_{13}\otimes\tau_{24}\simeq\omega_{\tilde{S}},
\end{equation}
where the last equality is a consequence of \eqref{eq:canbundletype1}.Thus, if we denote by $\gamma$ the generator of $\operatorname{Ker}\tilde{\pi}^*$, we get that
$$
\operatorname{Nm}(\tau_{13})\in\{\omega_S,\;\omega_S\otimes\gamma\} \subset \Pic^0(S).$$
Then we can use the $\Pic^0$ trick (Remark \ref{rmk:trick}) and find a $\beta\in\Pic^0(S)$ such that $\operatorname{Nm}(\tilde{\pi}^*\beta\otimes\tau_{13})$ is trivial.
\begin{lemma}\label{lem:classtau13}
In the above notation, the line bundle $\tilde{\pi^*}\beta\otimes\tau_{13}$ does not belong to  the image of $1-\tilde\sigma^*$
\end{lemma}
Before going forward with the proof, let us notice how, as an easy corollary, we get
\begin{cor}\label{cor:brbitype2}
If $S$ is of type 2, then the induced map $\pi_{\mathrm{Br}}:\mathrm{Br}(S)\rightarrow\mathrm{Br}(\tilde{S})$ is trivial. 
\end{cor}
\begin{proof}[Proof of Lemma \ref{lem:classtau13}]
We will show that the class of  $\tau_{13}$ in $H^2(\tilde{S},\mathbb{Z})$ is not in the image of $1-\tilde\sigma^*$. Denote by $[\tau_{ij}]$ the algebraic equivalence class of the line bundle $\tau_{ij}$. Then, by Proposition \ref{prop:torsionSerrano} and \eqref{eq:torsiontype1}, for every $L$ in $\Pic(\tilde{S})$ there are integers $n,$ $m$, and $h$, and $k$ such that
$$
c_1(L)=\frac{n}{2}\cdot a+m\cdot b+h\cdot[\tau_{13}]+k\cdot[\tau_{14}]
$$
Then it is easy to see that 
$$
(1-\tilde\sigma^*)c_1(L)=2h\cdot[\tau_{13}]+2k\cdot[\tau_{14}]=0.
$$
 But on the other side we have that $c_1(\tilde{\pi^*}\beta\otimes\tau_{13})=[\tau_{13}]$ is not trivial, thus $\tilde{\pi^*}\beta\otimes\tau_{13}$ cannot possibly lie in the image of $(1-\tilde\sigma^*)$, and the lemma is proved.
\end{proof}

\subsection{Bielliptic surface of type 3} In this paragraph we aim to show the following statement
\begin{thm}\label{thm:brbi}
If $S$ is a bielliptic surface of type 3, then the Brauer map $\tilde{\pi}_{\operatorname{Br}}:\operatorname{Br}(S)\rightarrow\operatorname{Br}(\tilde{S})$ induced by the cover $\tilde{\pi}:\tilde{S}\rightarrow S$, where $\tilde{S}$ is bielliptic of type 1, is injective.
\end{thm}
We will use \ref{prop:Beauville} and show that $\operatorname{Ker}(\operatorname{Nm})/\operatorname{Im}(1-\sigma^*)$ is trivial. There are two main key steps:
\begin{enumerate}
\item We first study the norm map when applied to numerically trivial line bundles; 
    \item then we prove that all the line bundles $L$ in $\operatorname{Ker}(\operatorname{Nm})$ are numerically trivial.
    
\end{enumerate}
\subsubsection{Norm of numerically trivial line bundles}
We will use the notation of Example \ref{ex:type3->type1}. Observe that we have the following diagram
\begin{equation}\label{eq:diagram}
\xymatrix{
\tilde{S}\ar[d]_{a_{\tilde{S}}}\ar[rr]^{\tilde{\pi}}&& S\ar[d]^{a_{{S}}}\\
A/G\ar[rr]_\varphi&&A/H
},\end{equation}
Where $G\simeq\mathbb{Z}/2\mathbb{Z}$,  and $H$ is $\mathbb{Z}/4\mathbb{Z}$. 
\begin{rmk}\label{rmk:isogeny}
Note that the bottom arrow, $\varphi$, is an isogeny of degree 2. As the vertical arrows are the Albanese maps of $\tilde{S}$ and $S$ respectively, we have that $\tilde{\pi}^*:\operatorname{Pic}^0(S)\rightarrow\operatorname{Pic}^0(\tilde{S})$ coincides with the isogeny dual to $\varphi$. In particular it is surjective.
\end{rmk}
Our first step in the study of the norm homomorphism for numerically trivial line bundles is to see how it behaves when applied to the generator of the torsion of $H^2(\tilde{S},\Z)$. In order to do that, we remark that the automorphism $\omega$ acts on $B[2]$ with at least one fixed point, the one corresponding to the identity element of $B$. Since $\omega$ has order 4, it cannot act transitively on the remaining three points on $B[2]$. Thus the action has at least two fixed points. We deduce that $\tilde{\sigma}$ acts on the set of the reduced multiple fibers by leaving fixed at least two of them, let us say $D_1$ and $D_2$. If the action were trivial, then we would have that all the line bundles $\tau_{ij}$ are invariant under the action of $\tilde{\sigma}$ and as a consequence they would be pullbacks of line bundles coming from $S$. We would deduce that all the torsion classes of $H^2(\tilde{S},\Z)$ are pullbacks of classes from $H^2(S,\Z)$, which is impossible. Thus we know that $\tilde{\sigma}$ exchanges $D_3$ and $D_4$. Then we can prove the following Lemma.

\begin{lemma}\label{lem:tauij} Let $n$ and $m$ be two integers. Then the norm of the line bundle $\tau_{13}^{\otimes n}\otimes\tau_{14}^{\otimes m}$ is zero if and only if $n$ and $m$ have the same parity. In addition we have that $\operatorname{Nm}(\tau_{13}^{\otimes n}\otimes\tau_{14}^{\otimes m})$ is not in $\Pic^0(S)$ if $n$ and $m$ are not congruent modulo 2.
\end{lemma}
\begin{proof}
Observe first of all that, thanks to the above discussion, the line bundle $\tau_{34}\simeq\tau_{13}\otimes\tau_{14}$ is invariant with respect to the action of $\tilde{\sigma}$. In particular we can write $\tau_{34}\simeq\tilde{\pi}^*\tau$ where $\tau$ is a line bundle on $S$ whose algebraic equivalence class is the only nontrivial class in $H^2(S,\Z)$.\par
Now, if $n$ and $m$ are both even, then $\tau_{13}^{\otimes n}\otimes\tau_{14}^{\otimes m}$ is the trivial line bundle, and there is nothing to prove. Otherwise, if $n$ and $m$ are odd, then
$$
\operatorname{Nm}(\tau_{13}^{\otimes n}\otimes\tau_{14}^{\otimes m})\simeq \operatorname{Nm}(\tau_{34})\simeq\tau^{\otimes 2}\simeq \sO_S.
$$
Conversely suppose that $n$ and $m$ are not congruent modulo 2. Up to exchanging $n$ and $m$
 we can assume that $m$ is even, while $n$ is odd. Then $\tau_{13}^{\otimes n}\otimes\tau_{14}^{\otimes m}\simeq\tau_{13}$. Again by \eqref{eq:pullbackofnorm} we get
 $$\tilde{\pi}^*\operatorname{Nm}(\tau_{13})\simeq\tau_{13}\otimes\tilde{\sigma}^*\tau_{13}\simeq\tau_{34}\simeq \tilde{\pi}^*\tau.
 $$
 We deduce that $\operatorname{Nm}(\tau_{13})$ is either equal to $\tau$ or to $\tau\otimes\omega_S^{\otimes 2}$. In any case it is not algebraically equivalent to zero and so the statement is proven.
 \end{proof}
\begin{rmk}\label{rmk:tau13} (a) Observe that $\tau_{34}$ is in the image of $1-\tilde{\sigma}^*$, as we have that $\tau_{34}\simeq\sO_{\tilde{S}}(D_3)\otimes\tilde{\sigma}^*\sO_{\tilde{S}}(-D_3)$.\par
(b) We will see in what follows that the different behavior of the norm map applied to torsion classes is what determines the contrast between the type 2 and type 3 bielliptic surfaces. In particular, the fact that the norm map of a torsion class is not necessarily algebraically trivial is what does not allow us to use Remark \ref{rmk:trick} in order to provide a non trivial class in $\operatorname{Ker}(\operatorname{Nm})/\operatorname{Im}(1-\tilde{\sigma}^*)$
\end{rmk}
Now we turn our attention to the elements of $\Pic^0(\tilde{S})$ whose norm is trivial. We will show that they never determine nonzero classes in $\operatorname{Ker}(\operatorname{Nm})/\operatorname{Im}(1-\tilde{\sigma}^*)$.
\begin{lemma}\label{lem:bipic0}Denote by $\operatorname{Nm}:\operatorname{Pic}(\tilde{S})\rightarrow
\operatorname{Pic}(S)$ the norm homomorphism.
Let $L\in\operatorname{Pic}^0(\tilde{S})$, such that $\operatorname{Nm}(L)=\sO_S$. Then the class of $L$ in $H^1(\mathbb{Z}/2\mathbb{Z}, \operatorname{Pic}(\tilde{S}))$ is trivial.
\end{lemma}
\begin{proof}
We have to show that such $L$ is in the image of the morphism $1-\tilde{\sigma}^*$.  By Remark \ref{rmk:isogeny}, we can write $L\simeq\tilde{\pi}^*M$ with $M\in\operatorname{Pic}^0(S)$. Then our assumption warrants that
$$\sO_S\simeq \operatorname{Nm}(L)\simeq M^{\otimes 2}.$$
We deduce that $M$ is a 2-torsion point in $\operatorname{Pic}^0(S)$. Now we know that $\operatorname{Pic}^0(S)[2]$ is a group scheme isomorphic to $\mathbb{Z}/2\mathbb{Z}\times\mathbb{Z}/2\mathbb{Z}$. Let $\gamma$ be the element $\omega_S^{\otimes2}\in\operatorname{Pic}^0(S)[2]$  then we can find $\beta\in \operatorname{Pic}^0(S)[2]$, $\beta$ nontrivial, such that
$$\operatorname{Pic}^0(S)[2]=\{\sO_S,\gamma,\beta,\gamma\otimes\beta\}.
$$
In particular, as $\tilde{\pi}^*\gamma\simeq\sO_{\tilde{S}}$,
\begin{equation}\label{eq:kernmbielliptic}
\operatorname{Ker}(\operatorname{Nm})\cap\operatorname{Pic}^0(\tilde{S})=\{\sO_{\tilde{S}},\tilde{\pi}^*\beta\}.
\end{equation}
 Now we aim at producing a line bundle $\alpha\in \operatorname{Pic}^0(\tilde{S})\cap\operatorname{Im}(1-\tilde\sigma^*)$, $\alpha\not\simeq \sO_{\tilde{S}}$. Thus we will have that $\operatorname{Pic}^0(\tilde{S})\cap\operatorname{Im}(1-\tilde\sigma^*)$ is a nontrivial subgroup of $\operatorname{Ker}(\operatorname{Nm})\cap\operatorname{Pic}^0(\tilde{S})$.  From \eqref{eq:kernmbielliptic} we deduce that
 $$
 \operatorname{Ker}(\operatorname{Nm})\cap\operatorname{Pic}^0(\tilde{S})=\operatorname{Pic}^0(\tilde{S})\cap\operatorname{Im}(1-\tilde\sigma^*)
 $$
 and so the statement.\par
 To this aim let $\overline{\epsilon}\in A':=A/G$ the image of the point $\epsilon\in A$ defining the involution $\tilde{\sigma}$ (see \eqref{eq:explicitsigma3}). Denote also by $p_0$ the identity element of $A'$; observe that by the construction of bielliptic surfaces $\overline{\epsilon}\neq p_0$. Consider the following line bundle on $\tilde{S}$: $$\alpha:=a_{\tilde{S}}^*(\sO_{A'}(p_0)\otimes t^*_{\overline{\epsilon}}\sO_{A'}(-p_0)).$$
 Clearly $\alpha$ is a nontrivial element in $\operatorname{Pic}^0(\tilde{S})$. In addition by \eqref{eq:explicitsigma3}  we see that 
 $$\alpha\simeq a_{\tilde{S}}^*(\sO_{A'}(p_0))\otimes \tilde\sigma^*a_{\tilde{S}}^*(\sO_{A'}(-p_0)), $$
 and therefore it is in the image of $1-\tilde\sigma^*$. Thus we can conclude.
\end{proof}
\subsubsection{Injectivity of the Brauer map} We are now ready to prove Theorem \ref{thm:brbi}. We will do so by showing the following statement.
\begin{prop}\label{prop:binumtri}
If $L\in\operatorname{Ker}(\operatorname{Nm})$ then $L$ is numerically trivial.
\end{prop}
Before proceeding with the proof, let us show how this implies Theorem \ref{thm:brbi}. Let $L$ be a line bundle in the kernel of the norm map. Then Proposition \ref{prop:binumtri} yields that $$L\simeq\alpha\otimes\tau_{13}^{\otimes n}\otimes\tau_{14}^{\otimes m}$$
for some positive integers $n$ and $m$, and for some $\alpha\in\operatorname{Pic}^0(\tilde{S}).$ Write again $\alpha\simeq\tilde{\pi}^*\beta$, and observe that if $n$ and $m$ are not congruent modulo 2, then by Lemma \ref{lem:tauij},we get 
$$
\operatorname{Nm}(L)\simeq\beta^{\otimes 2}\otimes \tau,
$$
which is not algebraically trivial. We deduce that $n$ and $m$ must have the same parity. Now we apply the first part of Lemma \ref{lem:tauij} and see that $\alpha\in \operatorname{Ker}(\operatorname{Nm}).$ In particular Lemma \ref{lem:bipic0} implies that $\alpha\in\operatorname{Im}(1-\tilde{\sigma}^*)$, and so the class of $L$ in $\operatorname{Ker}(\operatorname{Nm})/\operatorname{Im}(1-\tilde{\sigma}^*)$ is the same as the class of $\tau_{34}$. But Remark \ref{rmk:tau13}(a) tells us that the latter is trivial and so  Theorem \ref{thm:brbi} is proved.

\begin{proof}[Proof of Proposition \ref{prop:binumtri}]
Let $L$ in the kernel of the norm map. Lemmas \ref{lem:type3} and \ref{lem:type2} imply that $\tilde\pi^*\operatorname{Num}(S)$ is a sublattice of index 2 of $\operatorname{Num}(\tilde{S})$. In particular $L^{\otimes 2}$ is numerically equivalent to the pullback of a line bundle from $S$. Thus we can write
$$
L^{\otimes 2}\simeq\tilde{\pi}^*M\otimes\alpha\otimes\tau_{13}^{\otimes n}\otimes\tau_{14}^{\otimes m}$$
for some positive integers $n$ and $m$, and for some $\alpha\in\operatorname{Pic}^0(\tilde{S}).$ Again, by Remark \ref{rmk:isogeny}
 we can write $\alpha\simeq\tilde{\pi}^*\beta$ for some $\beta\in\operatorname{Pic}^0(S)$, and so, up to substituting $M$ with $M\otimes\beta$ we have that 
 $$
L^{\otimes 2}\simeq\tilde{\pi}^*M\otimes\tau_{13}^{\otimes n}\otimes\tau_{14}^{\otimes m}.$$
If we show that $M$ is numerically trivial we can conclude. Observe that
$$
\sO_S\simeq\operatorname{Nm}(L)\otimes\operatorname{Nm}(L)\simeq\operatorname{Nm}(L^{\otimes 2})\simeq M^{\otimes 2}\otimes\operatorname{Nm}(\tau_{13}^{\otimes n}\otimes\tau_{14}^{\otimes m})\simeq M^{\otimes 2}\otimes\tau^{\otimes (n+m)},
$$
where the last equality is a consequence of Lemma \ref{lem:tauij}. As $\tau$ is numerically trivial we conclude that the same is true for $M$.
\end{proof}
\section{The Brauer map to the canonical cover}\label{sec:abelian}
In this section we study the Brauer map $\pi_{\operatorname{Br}}:\operatorname{Br}(S)\rightarrow\operatorname{Br}(X)$ when $S$ is a bielliptic surface and $X$ is its canonical cover. Then there is an $n$ to 1 \'etale cyclic cover $\pi:X\rightarrow S$, where $n$ denotes the order of the canonical bundle $\omega_S$. Thus, as in the previous section, we can use Beauville's work \cite{Be2009} to study the kernel of the $\pi_{\operatorname{Br}}$ via the norm homomorphism $\operatorname{Nm}:\Pic(X)\rightarrow  \Pic(S)$. As in the other cases the Brauer group is trivial, we can assume that $S$ is of type 1, 2, 3, or 5. Recall that, independently from the case at hand,  there are two elliptic curves $A$ and $B$ such that $X$ is isogenous to $A\times B$. In what follows we will see that the geometry of the Brauer maps depends much on the geometry of $A\times B$, and in particular on whether there are isogenies between $A$ and $B$ or not. Throughout this section we will use the notation established in paragraph \ref{sub:cancover}.\par

\subsection{The norm of numerically trivial line bundles} Our first step will be proving the following proposition, which will allows us to study the norm map from a strictly numerical point of view. 
\begin{prop}\label{prop:brabpic0} Let $L\in\Pic^0(X)\cap\operatorname{Ker}(\operatorname{Nm})$. Then $L$ is in $\operatorname{Im}(1-\sigma^*)$.
\end{prop}
Before going any further we need to describe more precisely our setting and introduce some notation.\par
Observe first that, if we let as in \ref{sub:cancover} $p_A:X\rightarrow A/H$ and $p_B:X\rightarrow B/H$ the two elliptic fibrations of the abelian variety $X$, then $\Pic^0(X)$ is generated by $p_A^*\Pic^0(A/H)$ and $p_B^*\Pic^0(B/H)$, thus we can write any $L\in\Pic^0(X)$ as 
$p_A^*\alpha\otimes p_B^*\beta,
$ where $\alpha\in\Pic^0(A/H)$, $\beta\in\Pic^0(B/H)$. In this notation we have the following.

\begin{lemma}\label{lem:pbeta}
For every $\beta\in\Pic^0(B/H)$ we have that $p_B^*\beta$ is in the image of $1-\sigma^*$. In particular these line bundles are in the kernel of the norm homomorphism.
\end{lemma}
\begin{proof} We suppose first that $G$ is cyclic and so the group $H$ is trivial, and $X\simeq A\times B$. We proceed with a case by case analysis.\par
\emph{Type 1 case}. Since abelian varieties are divisible groups, there exist $\gamma\in\Pic^0(B)$ such that $\gamma^{\otimes 2}\simeq\beta$. Then by \eqref{eq:absigma} we have that
$$(1-\sigma^*)p_B^*\gamma\simeq p_B^*\gamma\otimes(\sigma^*p_B^*\gamma)^{-1}\simeq p_B^*2_B^*\gamma\simeq p_B^*\beta,
$$
and the statement is proven in this case.\par
\emph{Type 3 case}. In this case the $j$-invariant of $B$ is 1728 and there is an automorphism $\omega$ of $B$ of order 4. Consider the map $1-\omega:B\rightarrow B$. Since this is not trivial it is an isogeny, and in particular $(1-\omega)^*:\Pic^0(B)\rightarrow\Pic^0(B)$ is surjective. Let $\gamma\in\Pic^0(B)$ such that $(1-\omega)^*\gamma\simeq\beta$, then by \eqref{eq:absigma} we have
$$
(1-\sigma^*)p_B^*\gamma\simeq p_B^*(1-\omega)^*\gamma\simeq p_B^*\beta,
$$
and the statement is proven in this case.\par
\emph{Type 5 case}. This case is similar to the previous one in which instead of $\omega$ we use the automorphism $\rho$. We note that $(1-\rho):B\rightarrow B$ is non trivial, and so an isogeny. In particular the dual map $(1-\rho)^*:\Pic^0(B)\rightarrow\Pic^0(B)$ is surjective and we can find $\gamma$ such that $(1-\rho)^*\gamma\simeq \beta$. Again \eqref{eq:absigma} yields:
$$
(1-\sigma^*)p_B^*\gamma\simeq p_B^*(1-\rho)^*\gamma\simeq p_B^*\beta,
$$
and the statement is proven.\par
In order to conclude we need to analize the case of bielliptic surfaces of type 2. Under this assumption the group $H$ is not trivial but it is cyclic of order 2. Let $B':=B/H$ and observe that we have the following diagram
$$
\xymatrix{X\ar[d]_{p_B}\ar[rr]^\sigma&&X\ar[d]^{p_B}\\B'\ar[rr]^{-1_{B'}}&& B'}
$$
So let, as in the type 1 case, $\gamma\in\Pic^0(B')$ such that $2_{B'}^*\gamma\simeq \beta$, then we will have again that $(1-\sigma^*)p_B^*\gamma\simeq p_B^*\beta$ and the proof is concluded.
\end{proof}
Now let $L=p_A^*\alpha\otimes p_B^*\beta\in\Pic^0(X)$ such that $\operatorname{Nm}(L)\simeq\sO_S.$ Lemma \ref{lem:pbeta} implies that also $p_A^*\alpha$ is in the kernel of the norm homomorphism. In addition we have that the class of $L$ in $H^1(G,\Pic(A\times B))$ is just the class of $p_A^*\alpha$. We have a commutative diagram
\begin{equation*}
\xymatrix{
X\ar[d]_{p_A}\ar[rr]^{\pi}&& S\ar[d]^{a_{{S}}}\\
A/H\ar[rr]_\varphi&&A/G
},\end{equation*}
where the bottom arrow is an isogeny of degree $n$. In particular we can write $p_A^*\alpha\simeq\pi^*M$ with $M\in\Pic^0(S)$. In addition we have that
$$
\sO_S\simeq\operatorname{Nm}(p_A^*\alpha)\simeq M^{\otimes n},
$$
thus we have that
$$
p_A^*\left(\Pic^0(A/H)\right)\cap\operatorname{Ker}(\operatorname{Nm})=\pi^*\left(\Pic^0(S)[n]\right).
$$
It easy to see that the right-hand-side above is a group isomorphic to the cyclic group of order $n$.  Since $\operatorname{Im}(1-\sigma^*)$ is a subgroup of the kernel of the norm, if we provide an element of order $n$  in  $p_A^*\left(\Pic^0(A/H)\right)\cap\operatorname{Im}(1-\sigma^*)$ we would conclude that $$ p_A^*\left(\Pic^0(A/H)\right)\cap\operatorname{Im}(1-\sigma^*)=p_A^*\left(\Pic^0(A/H)\right)\cap\operatorname{Ker}(\operatorname{Nm})$$ and consequently the statement of Proposition \ref{prop:brabpic0}.
Let $p_0$ be the identity element of $A/H$, using the notation of \ref{eq:absigma} and \eqref{eq:absigmatype2} we set
$$
\gamma:=\begin{cases}
\sO_A(p_0)\otimes t_\tau^*(\sO_A(-p_0)),\quad&\text{if $S$ is of type 1,}\\
\sO_{A/H}(p_0)\otimes t_{\tau'}^*(\sO_{A/H}(-p_0)),\quad&\text{if $S$ is of type 2,}\\
\sO_A(p_0)\otimes t_\epsilon^*(\sO_A(-p_0)),\quad&\text{if $S$ is of type 3,}\\
\sO_A(p_0)\otimes t_\eta^*(\sO_A(-p_0)),\quad&\text{if $S$ is of type 5;}\\
\end{cases}
$$
where $\tau'$ is the image of $\tau$ under the isogeny $A\rightarrow A/H$.
Then $\gamma$ is a nontrivial element of $\Pic^0(A/H)$ with the desired property. In addition, by \eqref{eq:absigma} \eqref{eq:absigmatype2}, we have that $p_A^*\gamma\simeq (1-\sigma^*)p_A^*\sO_A(p_0),$ and so we can conclude. \hfill\qed\\

Now we are ready to start our investigation of the Brauer map $\pi_{\operatorname{Br}}:\operatorname{Br}(S)\rightarrow\operatorname{Br}(X)$. We first put ourselves in the special situation in which there are no nontrivial morphisms between $A$ and $B$.
\subsection{The Brauer map when the two elliptic curves are not isogenous}\label{5.2} If there are no isogenies between $A$ and $B$, the the lattice $\operatorname{Num}(X)$ has rank 2 and it is generated by the classes of the two fibers, $a_X$ and $b_X$. In addition, $\pi^*\operatorname{Num}(S)$ is a sublattice of $\operatorname{Num}(X)$ of index $n$. So, let $L$ be in the kernel of the norm map. We have that $L^{\otimes n}$ is numerically equivalent to the pullback of a line bundle from $S$. More precisely we can write
$$L^{\otimes n}\simeq \pi^*L'\otimes p_A^*\alpha\otimes p_B^*\beta\simeq\pi^*M\otimes p_B^*\beta,
$$
with $\beta\in\Pic^0(B/H)$. Lemma \ref{lem:pbeta} ensures that $\pi^*M$ is in the kernel of the norm map. In particular, $M$ is an $n$-torsion element in $\Pic(S)$. We deduce that it is numerically trivial, and so $L$ was numerically trivial to start with. Now we apply Proposition \ref{prop:brabpic0} and deduce the following statement.
\begin{thm}\label{thm:injective} If $S:=A\times B/G$ is a bielliptic surface such that the elliptic curves $A$ and $B$ are not isogenous, then the Brauer map to the canonical cover $\pi_{\text{Br}}:\operatorname{Br}(S)\rightarrow\operatorname{Br}(X)$ is injective.
\end{thm}
Before going to the next case, observe that if $S$ is a bielliptic surface of type 2, then  we have the following diagram
$$
\xymatrix{A\times B\ar[rr]^{\pi_{\tilde{S}}}\ar[d]_{\varphi}&&\tilde{S}\ar[d]^{\tilde{\pi}}\\
X\ar[rr]_{\pi_S}&&S.
}
$$
If $A$ and $B$ are not isogenous Theorem \ref{thm:injective}) above implies that the Brauer map induced by $\pi_S$ is injective. On the other side, the results of this paragraph imply that the Brauer map induced by $\pi_S\circ\varphi$ is trivial. Then the Brauer map induced by $\varphi$ cannot be injective and we have
\begin{cor}
If $\varphi:X\rightarrow Y$ is an isogeny of abelian varieties, the map $\varphi_{\operatorname{Br}}:\operatorname{Br}(Y)\rightarrow\operatorname{Br}(X)$ is not necessarily injective.
\end{cor}
\subsection{The Brauer map when the two elliptic curves are isogenous} Suppose now that $A$ and $B$ are isogenous.  Our first step will be to use the description the Picard group and  of the Neron--Severi of $A\times B$ that we outlined in \ref{subsec:NUMprod} in order to find the image of $1-\sigma^*$ and the numerical type of line bundles in the kernel of the Norm homomorphism when $S$ is a cyclic bielliptic surface. We begin with the following Lemma.
\begin{lemma}\label{lem:potentialclasses}
Suppose that $G$ is a cyclic group, so that $X\simeq A\times B$. If $L\in\Pic(A\times B)$ is in the kernel of the Norm map, then $c_1(L)=l(0,0,\varphi)$ for some isogeny $\varphi:B\rightarrow A$.
\end{lemma}
\begin{proof}
By the result of \ref{subsec:NUMprod} we have that $c_1(L)=l(m,n,\varphi)$ for two integers $n$ and $m$ and an isogeny $\varphi$. Suppose that $S$ is of type 1 and  $L$ is in the kernel of the norm map. We have that $L\otimes\sigma^*L$ is trivial. In particular $c_1(L\otimes\sigma^*L)$ is zero. But then we get the following
\begin{align*}
0&= c_1(L\otimes\sigma^*L)\\
&= c_1(L)+\sigma^*c_1(L)\\
&= l(m,n,\varphi)+l(m,n,-\varphi)=l(2m,2n,0).
\end{align*}
We conclude that $n=m=0$. At the same time, if $S$ is of type 5, we have that
$$
0=c_1(L)+\sigma^*c_1(L)+(\sigma^2)^*c_1(L)=l(3m,3n,0).
$$
Finally, if $S$ is of type 3 we get
$$0=c_1(L)+\sigma^*c_1(L)+(\sigma^2)^*c_1(L)+(\sigma^3)^*c_1(L)=l(4m,4n,0),
$$
so the statement is proven.
\end{proof}
We turn now our attention to the Brauer map in general and we study it by performing a case by case analysis  on the different type of bielliptic surfaces.
\subsubsection{Bielliptic surfaces of type 1} In this paragraph we study the Brauer map to the canonical cover of bielliptic surfaces of type 1. If $B$ does not have complex multiplication, we fix, once and for all, $\psi:B\rightarrow A$ be  a generating isogeny. Otherwise we fix $\psi_i:B\rightarrow A$, for $i=1,2$ two generators of $\Hom(B,A)$.  Our first step is to describe $(1-\sigma^*)\Pic(A\times B)$.
\begin{lemma}\label{image}
Let $S$ be a bielliptic surface of type 1, and consider $L\in (1-\sigma^*)\Pic(A\times B)$, then there exist three integers $m$, $h$ and $k$, and a line bundle $\beta\in\Pic^0(B)$ such that
$$
L\simeq\begin{cases}
L(P_\tau^{\otimes m},\beta,2h\cdot\psi) &\text{if $B$ does not have complex multiplication;}\\
L(P_\tau^{\otimes m},\beta, 2h\cdot\psi_1+2k\cdot\psi_2) &\text{if $B$ has complex multiplication.}\\
\end{cases}
$$
\end{lemma}
\begin{proof}
We do the complex multiplication case, the other is similar. Let $M\in\Pic(A\times B)$, then by the results of \ref{subsec:NUMprod} we have that $M\simeq L(M_A,M_B,h\cdot\psi_1+k\cdot\psi_2)$. We can write $M_A\simeq\sO_A(n\cdot p_0)\otimes\alpha$ and $M_B\simeq \sO_B(m\cdot q_0)\otimes\gamma$ for $q_0$ the identity element of $B$, some integers $n$ and $m$ and some topologically trivial line bundles $\alpha$ and $\gamma$. With this notation we find that
\begin{align*}\sigma^* & M\simeq L\left(t_\tau^*\sO_A(n\cdot p_0)\otimes\alpha,\sO_B(m\cdot q_0)\otimes\gamma^{-1}\otimes(-h\cdot\psi_1-k\cdot\psi_2)^*P_\tau,-h\cdot\psi_1-k\cdot\psi_2\right).
\end{align*}
Observe that as $\gamma$ ranges in all $\Pic^0(B)$ also $\beta:=\gamma^{\otimes 2}\otimes (h\cdot\psi_1+k\cdot\psi_2)^*P_\tau$ ranges in the whole $\Pic^0(B)$. In addition we have that
$$
(1-\sigma^*)M\simeq L\left(P_\tau^{\otimes n}, \beta, 2h\cdot\psi_1+2k\cdot\psi_2\right)$$
\end{proof}
\begin{rmk}\label{rmk:iff image}
It is not difficult to check that, for any two integers $h$ and $k$
$$
L(0,0,2h\cdot\psi_1+2k\cdot\psi_2)=L(0,0,h\cdot\psi_1+k\cdot\psi_2)\otimes\sigma^*L(0,0,h\cdot\psi_1+k\cdot\psi_2)^{-1},
$$
and so it is in $\operatorname{Im}(1-\sigma^*)$.
\end{rmk}
We are now ready to prove one of the main statements of this section:
\begin{thm}\label{thm:type1}
Suppose that $S$ is a bielliptic surface of type 1 whose canonical cover is $A\times B$ with $A$ and $B$ isogenous elliptic curves. Then the Brauer map to the canonical cover of $S$ is not injective if, and only if, one of the following mutually exclusive conditions is satisfied:
\begin{enumerate}
    \item the elliptic curve $B$ (and so $A$) does not have complex multiplication and $\psi^*P_{\tau}$ is trivial;
    \item the elliptic curve $B$ (and so $A$) has complex multiplication and we have that at least one of the following line bundles is trivial
    \begin{equation}\label{eq:lnbndl}L_1:=\psi_1^*P_{\tau},\quad L_2:=\psi_2^*P_\tau,\quad L_3:=(\psi_1+\psi_2)^*P_\tau\end{equation}
\end{enumerate}
\end{thm}

\begin{proof}
We deal with the complex multiplication case that is slightly more involved. The argument for the other case is very similar.\par Before explaining the details of our reasoning we would like to give, for the reader convenience, a quick outline of the proof. The key observation is that the assumption on the line bundles \eqref{eq:lnbndl} are equivalent to the triviality of the norm of one of the following invertible sheaves:
\begin{equation}\label{candidates}
M_1:=(1\times\psi_1)^*\sP_A,\;M_2:=\;(1\times\psi_2)^*\sP_A,\; M_3:=(1\times(\psi_1+\psi_2))^*\sP_A
\end{equation}
being topologically trivial. Therefore, if the assumptions are verified, we can use the $\Pic^0$ trick (Remark \ref{rmk:trick}) to provide an element in the kernel of the norm map. Such an element will give by construction a nontrivial class in $\operatorname{Ker}\operatorname{Nm}/\operatorname{Im}(1-\sigma^*).$ Conversely, if neither of the line bundles is trivial, then an element in the kernel of the norm map will be forced to be numerically equivalent to $(1\times 2\cdot \varphi)^*\sP_A$ for some isogeny $\varphi\in\Hom(B,A)$. Then we will apply Lemma \ref{image} and see that such a line bundle lies in $\operatorname{Im}(1-\sigma^*),$ so no element of $\Pic(A\times B)$ yields a nontrivial class in $\operatorname{Ker}\operatorname{Nm}/\operatorname{Im}(1-\sigma^*).$\par
Now, for the complete argument, observe first that  by \eqref{eq:pullbackofnorm} and the see-saw principle, it is easy to check  that, for every $\alpha$ in $\Pic^0(A)$ and every isogeny $\varphi:B\rightarrow A$, 
\begin{align}\label{eq:important type 1}
\begin{split}
    \pi^*\operatorname{Nm}((1\times\varphi)^*\sP_A\otimes p_A^*\alpha)&\simeq
    (1\times\varphi)^*\sP_A\otimes(1\times\varphi)^*\sP_A^{-1}\otimes\\
    &\quad\quad\otimes p_A^*\alpha^{\otimes 2}\otimes p_B^*\varphi^*P_{\tau}\\
    &\simeq p_A^*\alpha^{\otimes 2}\otimes p_B^*\varphi^*P_{\tau}.
\end{split}
\end{align}
Suppose first that one of the three line bundles in \eqref{eq:lnbndl} is trivial. To fix the ideas we can assume that $\psi_1^*P_\tau$ is trivial, the argument is identical in the other cases. Then by \eqref{eq:important type 1} we have that $\operatorname{Nm}((1\times\psi_1)^*\sP_A)$ is in the kernel of $\pi^*$, so in particular it is in $\Pic^0(S)$. We can therefore apply $\Pic^0$ trick (Remark \ref{rmk:trick}) and  find $\gamma\in\Pic^0(S)$ such that the norm of $(1\times\psi_1)^*\sP_A\otimes\pi^*\gamma$ is trivial. But by Lemma \ref{image} we have that $(1\times\psi_1)^*\sP_A\otimes\pi^*\gamma$ is not in the image of $1-\sigma^*$ and so it defines a non trivial class in
$\operatorname{Ker}\operatorname{Nm}/\operatorname{Im}(1-\sigma^*)$, and one direction of the statement is proven.\par
Conversely suppose that there is a line bundle $L$ on $X$ which identifies a non trivial class in $\operatorname{Ker}\operatorname{Nm}/\operatorname{Im}(1-\sigma^*)$. By Lemmas \ref{lem:potentialclasses} and \ref{image} we can write
$$
L\simeq(1\times h\cdot\psi_1+k\cdot\psi_2)^*\sP_A\otimes p_A^*\alpha\otimes p_B^*\beta,
$$
for two integers $h$ and $k$, and two topologically trivial line bundles $\alpha$
 and $\beta$. Not that  $h$ and $k$ cannot be both even, for otherwise Lemma \ref{image} and Remark \ref{rmk:iff image} yield that 
 $[L]=[p_A^*\alpha\otimes p_B^*\beta]\in\operatorname{Ker}\operatorname{Nm}/\operatorname{Im}(1-\sigma^*)$ which, by Proposition \ref{prop:brabpic0}, implies that  $[L]=0$ . Thus we can assume that one between $h$ and $k$ is odd. Then by Lemma \ref{lem:pbeta} and Lemma \ref{image} we have that $$L\simeq (1\times\psi_i)^*\sP_A\otimes p_A^*\alpha\otimes M,\quad\text{or}\quad L\simeq (1\times\psi_1+\psi_2)^*\sP_A\otimes p_A^*\alpha\otimes M,$$ with $M$ in $\operatorname{Im}(1-\sigma^*)$. In particular one of the following line bundles
 in \ref{candidates}  has trivial norm. We deduce by \eqref{eq:important type 1} that one of the line bundles in \eqref{eq:lnbndl} is trivial and the statement is proved.
 \end{proof}
 \begin{ex}\label{not count} \begin{enumerate}[leftmargin=0cm,itemindent=1cm,labelwidth=.5cm,labelsep=.1cm,align=left,label=(\alph*)]
     \item Suppose that $A\simeq B$. If $A$ does not have complex multiplication, then we can take $\psi=\pm1_A$. In particular we have that $\psi^*P_\tau$ is never trivial and the Brauer map is injective. 
     \item Suppose again that $A\simeq B$ and that the $j$-invariant of $A$ is 1728. Then $\End(A)\simeq\mathbb{Z}[i]$ and the multiplication by $i$ induces an automorphism $\omega$ of $A$ of order 4, and we can take $1_A$ and $\omega$ as generators of $\End(A)$. Suppose that $P_\tau$ is a fixed point\footnote[2]{For example we can identify $A$ with its dual and $\omega^*$ with $\omega$ and take  $\tau=(\frac{1}{2},\frac{1}{2})+\Lambda$, where $\Lambda=<1,i>$ $A\simeq\mathbb{C}/\Lambda$.} of the dual automorphism $\omega^*$. Then $(1_A+\omega)^*P_\tau$ is zero and the Brauer map is not injective.
     \item We can also use a similar argument to construct uncountably many Type 1 byelliptic surfaces with non injective Brauer map. Let $B$ any elliptic curve without complex  multiplication and chose $\theta$ a point of over 2 on $B$. Let $A:=B/\theta$ and $\psi:B\rightarrow A$ the quotient map. This is a degree 2 isogeny, so it is primitive and hence generating. If $\tau$ is the only point of order 2 in $\operatorname{Ker}\psi^*$, then we have that the data $A$, $\tau$, $B$ uniquely identify a Type 1 bielliptic surface which has a non injective Brauer map.
 \end{enumerate}
 
 \end{ex}
In order to complete our description of the Brauer map for type 1 bielliptic surfaces we need to give necessary and sufficient conditions for it to be trivial. To this aim we want to provide two distinct non-zero classes in $\operatorname{Ker}\operatorname{Nm}/\operatorname{Im}(1-\sigma^*).$
We can assume that the Brauer map is already non-injective, and so the condition of Theorem \ref{thm:type1} are satisfied. Suppose first that $B$ does not have complex multiplication. And consider $L$ in the kernel of the norm map, yielding a non trivial class in $\operatorname{Ker}\operatorname{Nm}/\operatorname{Im}(1-\sigma^*)$. Then, as before, we have that 
$$
L\simeq(1\times h\cdot\psi)^*\sP_A\otimes p_A^*\alpha\otimes p_B^*\beta.
$$
Again by Lemma \ref{image} we can assume that $h$ is odd, and the same result also yields that in $\operatorname{Ker}\operatorname{Nm}/\operatorname{Im}(1-\sigma^*)$ the class of $L$ and that of $(1\times \psi)^*\sP_A\otimes p_A^*\alpha$ are the same. Since $\psi^*P_\tau$ is trivial, \eqref{eq:important type 1} implies that  $(1\times \psi)^*\sP_A\otimes p_A^*\gamma$ is in the kernel of the norm map for some $\gamma\in\Pic^0(A)$. So $\operatorname{Nm}(p_A^*(\alpha\otimes\gamma^{-1}))\simeq\sO_S$ and as before $p_A^*(\alpha\otimes\gamma^{-1})$ lies in the image of $(1-\sigma^*)$. We deduce that, in $\operatorname{Ker}\operatorname{Nm}/\operatorname{Im}(1-\sigma^*)$,
$$
[L]=[(1\times\psi)^*\sP_A\otimes p_A^*\gamma]=[(1\times\psi)^*\sP_A\otimes p_A^*\delta],
$$
for every $\delta\in\Pic^0(A)$ such that $(1\times\psi)^*\sP_A\otimes p_A^*\delta$ is in the kernel of the norm homomorphism. In particular there is only one non-trivial element in $\operatorname{Ker}\pi_{\text{Br}}$.\par
Thus we can assume that $B$ has complex multiplication and that, as before, we have fixed $\psi_1$ and $\psi_2$ a system of generators for $\Hom(A,B)$. Suppose that only one among the line bundles \eqref{eq:lnbndl} is trivial, for example $L_1$, and as usual take $L$ in the kernel of the norm map. As before we can write $L\simeq M_i\otimes p_A^*\alpha\otimes M$  with $M$ in the image of $(1-\sigma^*)$ and $M_i$ one of the line bundles appearing in \eqref{candidates}. We deduce that $i=1$ and that the class of $L$ in $\operatorname{Ker}\operatorname{Nm}/\operatorname{Im}(1-\sigma^*)$ is equal to the class of $M_1\otimes p_A^*\gamma$  for every $\gamma\in\Pic^0(A)$ such that $\operatorname{Nm}(M_1\otimes p_A^*\gamma)$ is trivial. Thus, there is just one non-zero class and the Brauer map is again non trivial. Finally suppose that two (and so all) line bundles in \eqref{eq:lnbndl} are trivial. We have that both $M_1$ and $M_2$ are in the kernel of the norm map. In addition 
$$
M_1\otimes M_2^{-1}\simeq (1\times (\psi_1-\psi_2))^*\sP_A,
$$
which by Lemma \ref{image} is not in the image of $(1-\sigma^*)$. Therefore we deduce that they determine two different classes in $\operatorname{Ker}\operatorname{Nm}/\operatorname{Im}(1-\sigma^*)$, and hence the Brauer map is trivial.
We have thus proven the following statement
\begin{thm}\label{thm:t1trivial}
The Brauer map to the canonical cover of a type 1 bielliptic surface is trivial if, and only if, the elliptic curves $A$ and $B$ are isogenous, $B$ has complex multiplication, and  all the line bundles in \eqref{eq:lnbndl} are trivial.
\end{thm}
\begin{ex}
\begin{enumerate}[leftmargin=0cm,itemindent=1cm,labelwidth=.5cm,labelsep=.1cm,align=left,label=(\alph*)]
    \item 
    If $A\simeq B$ then the Brauer map is never trivial. Suppose otherwise that there are $\psi_1$ and $\psi_2$ generators of $\End(A)$ such that both $\psi_1^*P_\tau$ and $\psi_2^*P_\tau$ are zero. Then we can write $1_A=h\cdot\psi_1+k\cdot\psi_2$ and we would get that $P_\tau\simeq 1_A^*P_\tau$ is trivial, reaching an obvious contradiction.
   \item Let now $A\simeq\mathbb{C}/\Z[2i]$  and let $\tau$ the point $(0,i)+\Z[2i]$. The elliptic curve $B:=A/<\tau>$ has $j$-invariant 1728 and $\operatorname{Hom}(B,A)$ is generated by the isogenies $\psi_1:=\varphi_2$ and $\psi_2:=\varphi_2\circ\lambda_B$, where $\varphi_2:B\rightarrow A$ denotes the isogeny induced by multiplication by $2$ (see Example \ref{ex:computation} in the Appendix). Observe that  $$\varphi_2^*(P_\tau)\simeq\varphi_2^*(\sO_A(\tau-p_0)\simeq\sO_A(\varphi_2(\tau)-\varphi(p_0))\simeq\sO_B$$
   Thus we have that $\psi_1^*P_\tau\simeq\psi_2^*P_\tau\simeq\sO_B$ and the Brauer map is trivial.

\end{enumerate}

\end{ex}

\subsubsection{Bielliptic surfaces of type 3} Let now $S$ be a bielliptic surface of type 3. Then the canonical cover of $S$ is isomorphic to $A\times B$ with $j(B)=1728$ and multiplication by $i$ induces and automorphism $\omega$ of $B$ of order 4, $\omega$. By the discussion in \ref{subsec:NUMprod}, it is possible to find a generating isogeny $\psi$ such that 
$$
\operatorname{Num}(X)=\left<l(1,0,0),\;l(0,1,0)\;l(0,0,\psi),\; l(0,0,\psi\circ\omega)\right>,
$$
We fix, once and for all, such a $\psi$ and prove the following Lemma, which yields a precise description of $(1-\sigma^*)\Pic(X)$.
\begin{lemma}\label{lem:image3}
Let $\varphi:B\rightarrow A$ and isogeny. Then there are two integers $h$ and $k$
 such that $\varphi=h\cdot\psi+k\cdot\psi\circ\omega$. Then the line bundle $(1\times\varphi)^*\sP_A\in \operatorname{Im}(1-\sigma^*)$ if and only if h+k is even. \end{lemma}
\begin{proof}
Let $T:\operatorname{Hom}(B,A)\rightarrow \Hom(B,A)$ the linear operator obtained by pre composing by $(1_B-\omega)$. Then, using that $\omega^2=-1_B$, it is not difficult to show that an isogeny $\varphi$ as in the statement is in the image of $T$ if, and only if, $h+k$ is even. Hence if $h+k$ is not even, $(1\times\varphi)^*\sP_A\notin \operatorname{Im}(1-\sigma^*)$.\par
Suppose now that $\varphi=h\cdot\psi+k\cdot\psi\circ\omega$ with $h+k$ an even number. Then, by the above argument, we can find an isogeny $\gamma:B\rightarrow A$ such that $\varphi=\gamma\circ(1_B-\omega).$ Then we have 
\begin{equation*}
\begin{aligned}
(1-\sigma^*)(1\times\gamma)^*\sP_A &\simeq (1\times\gamma\circ(1_B-\omega))^*\sP_A\otimes p_B^*\omega^*\gamma^*P_{\epsilon}^{-1}\\
&\simeq(1\times \varphi)^*\sP_A\otimes p_B^*\omega^*\gamma^*P_{\epsilon}^{-1}.
\end{aligned}
\end{equation*}
By Lemma \ref{lem:pbeta}, elements of the form $p_B^*\beta$ with $\beta\in \Pic^0(B)$ are in the image of $(1-\sigma^*)$, so we conclude that $(1\times\varphi)^*\sP_A\in \operatorname{Im}(1-\sigma^*)$.
\end{proof}
\begin{rmk}\label{rmk:oneclass}
Observe that this Lemma implies easily that the quotient $\operatorname{Hom}(B,A)/\operatorname{Im}(1-\sigma^*)$, where we are identifying $\operatorname{Hom}(B,A)$ with the corresponding subgroup of $\operatorname{Num}(A\times B)$, is  cyclic generated by the coset $(1_A\times\psi)^*\sP_A+\operatorname{Im}(1-\sigma^*)$.
\end{rmk}
Now we are ready to start studying the kernel for the Brauer map $\pi_{\operatorname{Br}}:\operatorname{Br}(S)\rightarrow\operatorname{Br}(X)$.  Our main result is the following
\begin{thm}\label{thm:type3}
Let $S$ is a bielliptic surface of type 3 with canonical cover $A\times B$ such that $A$ and $B$ are isogenous. Then the Brauer map to the canonical cover is identically zero if, and only if, $(1_B+\omega)^*\psi^*{P_{2\epsilon}}$ is trivial
\end{thm}
\begin{proof}
For any isogeny $\varphi:B\rightarrow A$, $\alpha\in\Pic^0(A)$ and $\beta\in\Pic^0(B)$, using that the norm of $p_B^*\beta$ is trivial by Lemma \ref{lem:pbeta}, we have that 
\begin{equation}\label{eq:type3}
\begin{aligned}
    \pi^*\operatorname{Nm}((1\times\varphi)^*\sP_A\otimes p_A^*\alpha\otimes p_B^*\beta) &\simeq (1\times \varphi)^*\sP_A\otimes p_A^*\alpha\otimes\\
     & \quad\;\: (1\times \varphi\circ \omega)^*\sP_A\otimes p_B^*\omega^*\varphi^*P_\epsilon\otimes p_A^*\alpha\otimes\\
     & \quad\;\: (1\times -\varphi)^*\sP_A\otimes p_B^*(-1_B)^*\varphi^*P_{2\epsilon}\otimes p_A^*\alpha\otimes\\
     & \quad\;\: (1\times -\varphi\circ\omega)^*\sP_A\otimes p_B^*(-\omega)^*\varphi^*P_{3\epsilon}\otimes p_A^*\alpha\otimes\\
     &\simeq p_A^*\alpha^{\otimes 4}\otimes p_B^*(1+\omega)^*\varphi^*P_{2\epsilon}.
    \end{aligned}
\end{equation}
Suppose that $(1_B+\omega)^*\psi^*P_{2\epsilon}\simeq\sO_B$. Since $P_{2\epsilon}$ is a two torsion point, this is equivalent to asking that $(1_B-\omega)^*\psi^*P_{2\epsilon}$ is also trivial. Then \eqref{eq:type3} implies that the norms of $(1\times\psi)^*\sP_A$ and of $(1\times\psi\circ\omega)^*\sP_A$ lie in $\Pic^0(S)$. Then using the $\Pic^0$-trick (Remark \ref{rmk:trick}) and Lemma \ref{lem:image3} we can find a non zero class in $\operatorname{Ker}\operatorname{Nm}/\operatorname{Im}(1-\sigma^*),$ and the Brauer map is trivial.\par
Conversely, let $L$ be a line bundle defining a nontrivial class in $\operatorname{Ker}\operatorname{Nm}/\operatorname{Im}(1-\sigma^*)$. Then as we did in the case of type 1 surfaces, we can write
$$
L\simeq (1\times h\cdot\psi+k\cdot\psi\circ\omega)^*\sP_A\otimes p_A^*\alpha\otimes p_B^*\beta
$$
with $\alpha$ and $\beta$ in $\Pic^0(A)$ and $\Pic^0(B)$.  Lemma \ref{lem:type3} implies that the integer $h+k$ is odd or we would have that $p_A^*\alpha$ is in the kernel of the norm map, and consequently, by Proposition \ref{prop:brabpic0}, $L \in \operatorname{Im}(1-\sigma^*$) . Thus we can write
$$L\simeq M\otimes M'$$
where $M'$ is in the image of $1-\sigma^*,$  and $M$ is numerically equivalent to $(1\times\psi)^*\sP_A$ (this is a consequence of Lemma \ref{lem:image3} and Remark \ref{rmk:oneclass}).
We deduce that $M$ is in the kernel of the norm map. But then \eqref{eq:type3} implies that $
(1+\omega)^*\psi^*P_{2\epsilon}$ is trivial, proving the statement.
\end{proof}
\begin{ex}
Suppose that $A\simeq B$, so we can take $\psi=1_A$. If $P_{2\epsilon}$ is a fixed point of $\omega$, then we have that $\sP_A$ yields a nontrivial element in $\operatorname{Ker}\operatorname{Nm}/\operatorname{Im}(1-\sigma^*).$ Conversely, if $P_{2\epsilon}$ is not a fixed point of $\omega$ we will have that the Brauer map is injective.
\end{ex}
\subsubsection{Bielliptic surfaces of type 5}  Let $S$ be a bielliptic surface of type 5. We will solve this case in a similar fashion as for bielliptic surfaces of type 3. In the type-5 case, the canonical cover is isomorphic to an abelian surface $A\times B$ with $j(B)=0$. As already seen, $B$ admits an automorphism $\rho$ of order 3 such that $\rho^2+\rho+1=0$. Again, thanks to Theorem \ref{thm:injective} we need to study only the case in which $A$ and $B$ are isogenous. Also in this case, by the results of \ref{subsec:NUMprod}, there is generating isogeny $\psi:B\rightarrow A$ such that 
$$
\operatorname{Num}(X)=\left<l(1,0,0),\;l(0,1,0)\;l(0,0,\psi),\; l(0,0,\psi\circ\rho)\right>.
$$ With this notation, we prove a statement analogous to Lemma \ref{lem:image3}:

\begin{lemma}\label{lem:image5} 
Let $\varphi:B\rightarrow A$ and isogeny. Then there are two integers $h$ and $k$
 such that $\varphi=h\cdot\psi+k\cdot\psi\circ\rho$. If $h+k$ is not divisible by 3, then $(1\times\varphi)^*\sP_A\notin \operatorname{Im}(1-\sigma^*)$. Conversely if 3 divides $h+k$, then $(1\times\varphi)^*\sP_A\otimes p_B^*\beta\in \operatorname{Im}(1-\sigma^*)$, for every $\beta\in\Pic^0(B)$.\end{lemma}

\begin{proof} The argument is completely analogous to the proof of Lemma \ref{lem:image3}, after observing that, if $T:\operatorname{Hom}(B,A)\rightarrow\operatorname{Hom}(B,A)$ is the operator defined by pre composing with $1_B-\rho$, then the image of $T$ are exactly the homomorphism $h\cdot\psi+k\cdot\psi\circ\rho$ such that 3 divides $k+h$.  
\end{proof}
\begin{rmk}\label{rmk:oneclass2} This Lemma implies easily that the quotient of the Hom-part of $\operatorname{Num}(A\times B)$ by the action of $1-\sigma^*$ is isomorphic to $\Z/3\Z$ with elements  $(1_A\times\psi)^*\sP_A+\operatorname{Im}(1-\sigma^*)$ and $(1_A\times\psi+\psi\circ\rho)^*\sP_A+\operatorname{Im}(1-\sigma^*)=(1_A\times2\cdot\psi)^*\sP_A+\operatorname{Im}(1-\sigma^*)$.
\end{rmk}
We will also need the following statement:
\begin{lemma}\label{peta}
Let $B$ an elliptic curve with $j$-invariant 0 and $\beta$ an element $\Pic^0(B)$. Consider the following line bundles
$$
P_1:=(2\cdot\rho+1_B)^*\beta,\quad,P_{\rho}:=(2\cdot\rho+1_B)^*\rho^*\beta,\quad\text{and}\quad P_{1+\rho}:= (2\cdot\rho+1_B)^*(1_B+\rho)^*\beta.
$$
If one of them is trivial then they are all trivial.
\end{lemma}
\begin{proof}
Observe first that $(2\cdot\rho+1_B)^*\rho^*\beta\simeq \rho^*(2\cdot\rho+1_B)^*\beta$. Since $\rho$ is an automorphism the triviality of $P_\rho$ is equivalent to the triviality of $P_1$. In addition as $P_{1+\rho}\simeq P_1\otimes P_\rho$ we have that if $P_1$ and $P_\rho$ are both trivial, then also $P_{1+\rho}$ is trivial. It remains to show that if $P_{1+\rho}\simeq\sO_B,$ then also $P_1$ and $P_\rho$ are trivial. We not that $P_{1+\rho}\simeq\sO_B$ if, and only if,  $\rho^*P_{1+\rho}\simeq\sO_B$. On the other side we have
\begin{align*}
    \rho^*P_{1+\rho} &\simeq \rho^*(2\cdot\rho+1_B)^*(1_B+\rho)^*\beta\simeq \rho^*(\rho-1_B)^*\beta\simeq(-2\cdot\rho-1_B)^*\beta\simeq P_1^{-1}.
\end{align*}
We conclude that the triviality of $P_{1+\rho}$ is equivalent to the triviality of $P_1$ as required by the statement.
\end{proof}
Now we are ready to prove the main result of this paragraph:
\begin{thm}\label{thm:type5}
Let $S$ be an bielliptic surface of type 5 such that the two elliptic curves $A$ and $B$ are isogenous. Let $\psi$ be a generating isogeny, then we have that the Brauer map $\pi_{\operatorname{Br}}:\operatorname{Br}(S)\rightarrow\operatorname{Br}(A\times B)$ is trivial if, and only if, the line bundle
$(2\cdot\rho+1_B)^*\psi^*P_{\eta}\simeq \sO_B$.
\end{thm} 
\begin{proof}
The argument is really similar to what happens for type 3 bielliptic surfaces. We first note that, for any isogeny $\varphi: B\rightarrow A$, and every $\alpha$ and $\beta$ in $\Pic^0(A)$ and $\Pic^0(B)$ respectively, we have that 
\begin{equation}\label{eq:type5}
\begin{aligned}
    \pi^*\operatorname{Nm}((1\times\varphi)^*\sP_A\otimes p_A^*\alpha\otimes p_B^*\beta) &\simeq p_A^*\alpha^{\otimes 3}\otimes p_B^*(2\cdot\rho+1_B)^*\varphi^*P_{\eta}.
 \end{aligned}
\end{equation}
Suppose first that $(2\cdot\rho+1_B)^*\psi^*P_{\eta}$ is trivial. Then \eqref{eq:type5} ensures that the norm of $ M_1:= (1\times\psi)^*\sP_A$ is topologically trivial. By Lemma \ref{lem:image5} we know that no line bundle numerically equivalent to $M_1$ is in the image of $1-\sigma^*$. Thus we use the Remark \ref{rmk:trick} to provide an element in $\operatorname{Ker}\operatorname{Nm}$ inducing a non trivial class in $\operatorname{Ker}\operatorname{Nm}/\operatorname{Im}(1-\sigma^*)$.\par
Conversely, assume that $L$ is a line bundle in $\operatorname{Ker}\operatorname{Nm}$ whose class in $\operatorname{Ker}\operatorname{Nm}/\operatorname{Im}(1-\sigma^*)$ is not trivial. As before we can write
$$L\simeq (1\times h\cdot\psi+k\cdot\psi\circ\rho)^*\sP_A\otimes p_A^*\alpha\otimes p_B^*\beta.$$
We apply Lemma \ref{lem:image5} and write $L\simeq M\otimes M'$ with $M'\in\operatorname{Im}(1-\sigma^*)$ and $M$ a line bundle numerically equivalent to one of the following
\begin{equation}\label{eq:lnbndl5}
M_1 :=(2\cdot\rho+1_B)^*\psi^*P_{\eta},\quad \text{and}\quad M_{1+\rho}:=(2\cdot\rho+1_B)^*(1+\rho)^*\psi^*P_{\eta}.
\end{equation}
Clearly $M$ is in the kernel of the norm map, which, by \eqref{eq:type5} implies that one among the following is trivial:
$$
P_1:=(2\cdot\rho+1_B)^*\psi^*P_\eta,\quad \text{and}\quad P_{1+\rho}:= (2\cdot\rho+1_B)^*(1_B+\rho)^*\psi^*P_\eta.
$$
We conclude by applying Lemma \ref{peta} and deducing that $P_1\simeq\sO_B$.
\end{proof}
\begin{ex}
Suppose that $A\simeq B$. Note that the isogeny $\varphi:=(2\cdot\rho+1_B):B\rightarrow$ has degree 3, and its kernel is contained in $B[3]$ which has order 9. If $\eta$ is in the kernel of $\varphi$ then the bielliptic surface obtained by the action of $\sigma(x,y)=(x+\eta,\rho(y))$ has trivial Brauer map. Otherwise the Brauer map is injective.
\end{ex}

\subsubsection{Bielliptic surfaces of type 2} We kept last the bielliptic surfaces of type two since for them we need an \emph{ad hoc} argument. Let therefore $S$ be a bielliptic surface of type 2 and denote by $X$ its canonical cover. Then $X\simeq A\times B/<t_{(\theta_1,\theta_2)}>$ for two elliptic curves $A$ and $B$ and $\theta_1$ and $\theta_2$ points of order 2 in $A$ and $B$ respectively. Let us fix generators for $\Hom(B,A)$: if $B$ does not have complex multiplication then $\Hom(B,A)=<\psi>$ with $\psi:B\rightarrow A$ an isogeny; otherwise there are two isogenies $\psi_1, \psi_2:B\rightarrow A$ such that $\Hom(B,A)=<\psi_1,\psi_2>.$ Our goal is to prove the following statement.
\begin{thm}\label{thm:type 2}
In the above notation the Brauer map $\pi_{\operatorname{Br}}:\operatorname{Br}(S)\rightarrow\operatorname{Br}(X)$ is not injective if, and only if, one of the following conditions is satisfied:
\begin{enumerate}
    \item the elliptic curve $B$ does not have complex multiplication and  either $\psi(\theta_2)$ is not the identity element of $A$ or $\psi^*P_{\theta_1}$ is not trivial.
    \item the elliptic curve $B$ has complex multiplication and not all of the following elements are the identity element in the elliptic curve they belong to
    $$\left\{\psi_1(\theta_2),\;\psi_2(\theta_2),\;\psi_1^*P_{\theta_1},\psi_2^*P_{\theta_1},\; (\psi_1+\psi_2)(\theta_2),\;(\psi_1+\psi_2)^*(P_{\theta_1})\right\}
    $$
    \end{enumerate}
\end{thm}
Before proceeding with the proof we need to set up some notation. Recall that we have the following diagram
$$
\xymatrix{A\times B\ar[d]_{\phi}\ar[rr]^{\pi_{\tilde{S}}}&&\tilde{S}\ar[d]^{\tilde{\pi}}\\
X\ar[rr]_{\pi_S}&&S}
$$
where $\tilde{S}$ is a bielliptic surface of type 1.
We have that $S\simeq X/\sigma$, $\tilde{S}\simeq A\times B/\tilde{\sigma}$ and $X\simeq A\times B/\Sigma$, where $\Sigma$  denotes the translation $t_{(\theta_1,\theta_2)}$. We are going to deal just with the case in which $B$ hax complex multiplication. The proof in the other case will be identical, provided that one drops one of the two generators. We first observe the following fact:
\begin{lemma}\label{lem:Sigmainv}
In the notation above suppose that $B$ has complex multiplication and let $L_i$ be the line bundle $(1\times\psi_i)^*\sP_A$, for $i=1,2$. Then the conditions of Theorem \ref{thm:type 2} are satisfied if, and only if,  for every $\gamma\in\Pic^0(A\times B)$ one of the following line bundles is not $\Sigma$-invariant: 
\begin{equation}\label{eq:type2inv}
    L_1\otimes \gamma,\; L_2\otimes\gamma, L_1\otimes L_2\otimes\gamma.
\end{equation}
\end{lemma}
\begin{proof}
By see-saw, it is easy to see that
\begin{align*}
\Sigma^*[(1\times\psi_i)^*\sP_A\otimes\gamma]&\simeq(1\times\psi_i)^*\sP_A\otimes\gamma\otimes p_A^*P_{\psi_i(\theta_2)}\otimes p_B^*\psi_i^*P_{\theta_1},\\
\Sigma^*[(1\times(\psi_1+\psi_2))^*\sP_A\otimes\gamma]&\simeq(1\times(\psi_1+\psi_2))^*\sP_A\otimes\gamma\\ &\quad\quad\otimes p_A^*P_{\psi_1+\psi_2(\theta_2)}\otimes p_B^*(\psi_1+\psi_2)^*P_{\theta_1};
\end{align*}
the statement follows directly.
\end{proof}
\begin{proof}[Proof of the sufficiency of the conditions of the Theorem \ref{thm:type 2}]
Suppose that the conditions of the statement are satisfied. Then, by Lemma \ref{lem:Sigmainv}, one of the line bundles \eqref{eq:type2inv} is not $\Sigma$-invariant. Suppose first that $L_1\otimes\gamma$ is not $\Sigma$-invariant for every topologically trivial $\gamma$. Thus we have that $l(0,0,\psi_1)$ is not in $\phi^*\operatorname{Num}(X)$. 
We deduce that
\begin{equation}\label{eq:type2num}
2\cdot\psi_1\notin(1-\tilde{\sigma})^*\phi^*\operatorname{Num}(X).
\end{equation}
Otherwise we would have
\begin{align*}
  2\cdot\psi_1&= (1-\tilde{\sigma})^*\phi^*\varphi\\
  &=(1-\tilde{\sigma})^*(h\cdot\psi_1+k\cdot\psi_2)\\
  &=2h\cdot \psi_1+2k\cdot\psi_2.
\end{align*}
Therefore $h=1$, $k=0$ and $\phi^*\varphi=\psi_1$, contradicting our previous conclusion.
\indent Now consider the line bundle $L:=\operatorname{Nm}_\phi((1\times\psi_1)^*\sP_A)$. We want to show that there is $\beta\in\Pic^0(X)$ such that $\operatorname{Nm}_{\pi_S}(L\otimes\beta)$ is trivial. We use the functoriality of the norm map (Proposition \ref{prop:functorial}) and we obtain that $$\operatorname{Nm}_{\pi_S}(L)\simeq \operatorname{Nm}_{\tilde{\pi}}\circ\operatorname{Nm}_{\pi_{\tilde{S}}}((1\times \psi_1)^*\sP_A).$$
Observe that by \eqref{eq:important type 1} we have that $\pi_{\tilde{S}}^*\operatorname{Nm}_{\pi_{\tilde{S}}}((1\times\psi_1)^*\sP_A)$ is numerically trivial. Therefore we have that $\operatorname{Nm}_{\pi_{\tilde{S}}}((1\times\psi_1)^*\sP_A)$ is itself numerically trivial. This implies that $$\operatorname{Nm}_{\tilde{\pi}}\circ\operatorname{Nm}_{\pi_{\tilde{S}}}((1\times \psi)^*\sP_A)\in\Pic^0(S).$$ In fact if we have that $\operatorname{Nm}_{\pi_{\tilde{S}}}((1\times\psi_1)^*\sP_A):=\alpha\in\Pic^0(\tilde{S})$ then we write $\alpha\simeq\tilde{\pi}^*\gamma$ and we have that 
$$\operatorname{Nm}_{\pi_S}(L)\simeq \operatorname{Nm}_{\tilde{\pi}}\circ\operatorname{Nm}_{\pi_{\tilde{S}}}((1\times \psi_1)^*\sP_A)\simeq\gamma^{\otimes 2}.$$
On the other side if $\operatorname{Nm}_{\pi_{\tilde{S}}}((1\times\psi_1)^*\sP_A):=T$ a numerically trivial but not algebraically trivial line bundle, then as in \eqref{eq:tau13} we have that $\operatorname{Nm}_{\tilde{\pi}}(T)$ is topologically trivial. Thus, as before, we obtain $\beta$ such that $\operatorname{Nm}_{\pi_S}(L\otimes\beta)\simeq\sO_S$ via the $\Pic^0$ trick (Remark \ref{rmk:trick}).\par
In order to determine the non injectivity of the Brauer map we have to ensure that $L\otimes \beta$ is not in $\operatorname{Im}(1-\sigma^*)$. Suppose that this were not the case, and consider the following commutative diagram
$$
\xymatrix{A\times B\ar[rr]^{\tilde{\sigma}}\ar[d]_\phi&& A\times B\ar[d]^\phi\\
X\ar[rr]_\sigma&&X.}
$$
Then $c_1(\phi^*L)\in(1-\tilde{\sigma})^*\phi^*\operatorname{Num}(X)$. However the properties of the norm (see \eqref{eq:pullbackofnorm}) ensures that $c_1(\phi^*L)=l(0,0,2\cdot\psi_1)$, thus we would have that  $l(0,0,2\cdot\psi_1)\in\phi^*\operatorname{Num}(X)$, contradicting \eqref{eq:type2num}.\par
If $L_2\otimes \gamma$ is not $\Sigma$-invariant for every $\gamma\in\Pic^0(A\times B)$, then we proceed as before by exchanging the role of $\psi_1$ and $\psi_2$. Thus, it remain only to see what happen if $L_1\otimes L_2\otimes \gamma$ is not $\Sigma$-invariant for every $\gamma$. In this case we will have that $l(0,0,\psi_1+\psi_2)\notin\phi^*\operatorname{Num}(A\times B)$, and so either $l(0,0,\psi_1)$ or $l(0,0,\psi_2)$ are not in the image of $\phi^*$. Without loss of generality we can assume the first. Then we will still have \eqref{eq:type2num} and we can repeat the above argument.

\end{proof}
In order to complete the proof of Theorem \ref{thm:type 2} we need to show that if all $(1\times\psi_1)^*\sP_A$, $(1\times\psi_2)^*\sP_A$, and $(1\times(\psi_1+\psi_2))^*\sP_A$ are $\Sigma$-invariant then the Brauer map to $X$ is injective. Observe that, under this assumptions, we can write 
$$
(1\times\psi_1)^*\sP_A\simeq \phi^*L_1,\quad (1\times\psi_2)^*\sP_A\simeq\phi^* L_2,\quad\text{and}\quad (1\times(\psi_1+\psi_2))^*\sP_A\simeq\phi^* L_3.
$$
for some line bundles $L_1$, $L_2$, and $L_3$ in $\Pic(X)$. Then for $\alpha \in \Pic^0(X)$ write $\phi^*\alpha\simeq p_A^*\alpha_1\otimes p_B^*\alpha_2$ we have
\begin{align*}
\phi^*(\pi_S^*\operatorname{Nm}_{\pi_S}(L_i\otimes\alpha))&\simeq \phi^*(L_i\otimes\alpha\otimes\sigma^*(L_i\otimes \alpha))\simeq p_A^*\alpha_1^{\otimes 2}\otimes p_B^*(\psi_1^*P_\tau),\\
\phi^*(\pi_S^*\operatorname{Nm}_{\pi_S}(L_1\otimes L_2\otimes\alpha ))&\simeq \phi^*(L_1\otimes L_2\otimes\alpha\otimes\sigma^*(L_1\otimes L_2\otimes\alpha))\\
&\simeq p_A^*\alpha_1^{\otimes 2}\otimes p_B^*(\psi_1^*P_\tau\otimes\psi_2^*P_\tau);
\end{align*}
where, in both cases, the  last equality is again given by \eqref{eq:important type 1}.
Observe that neither the $\psi_i$'s nor $\psi_1+\psi_2$ can factor through the multiplication by 2 isogeny, or we would have that $\psi_1$ and $\psi_2$ cannot generate $\Hom(B,A)$. In particular,   we cannot have that neither $\psi_i^*P_\tau$ nor $(\psi_1+\psi_2)^*P_\tau$  can be trivial. We deduce that
\begin{align*}
\phi^*(\pi_S^*\operatorname{Nm}_{\pi_S}(L_i\otimes\alpha))&\not\simeq\sO_{A\times B},\\
    \phi^*(\pi_S^*\operatorname{Nm}_{\pi_S}(L_1\otimes L_2\otimes\alpha))&\not\simeq\sO_{A\times B}.
\end{align*}
In particular  we obtained the following lemma:
\begin{lemma}\label{lem:notker}
In the above notation, if the conditions of Theorem \ref{thm:type 2} are not satisfied, then line bundles numerically equivalent to $L_i$ or $L_1\otimes L_2$ are not in the kernel of the norm map $\operatorname{Nm}_{\pi_S}$.
\end{lemma}
Before going further we need an intermediate step:
\begin{lemma}\label{lem:type2im}
For any integer $n$, $L_i^{\otimes 2n}$ and $(L_1\otimes L_2)^{\otimes 2n}$ are in $\operatorname{Im}(1-\sigma^*)$
\end{lemma}
\begin{proof} Obviously it is enough to show that $L_i^{\otimes 2}$ is in the image of $(1-\sigma^*)$. To this aim, we pull $L_i\otimes\sigma^*L_i$ back to $A\times B$ and apply \eqref{eq:important type 1}. We see that $$\phi^*(L_i\otimes\sigma^*(L_i))\in p_B^*\Pic^0(B)\subseteq A\times B,$$
and we deduce that $\gamma:=L_i\otimes\sigma^*(L_i)$ is a  line bundle in $p_B^*\Pic^0(B/H)$. By \ref{lem:pbeta} we know that $\gamma\in\operatorname{Im}(1-\sigma^*)$. Thus we can write
$$
L_i^{\otimes 2}\simeq \gamma\otimes\sigma^*L_i\otimes L_i^{-1}.
$$
\end{proof}
\begin{proof}[Conclusion of the Proof of Theorem \ref{thm:type 2}]

Let $M$ is a line bundle such that $\operatorname{Nm}_{\pi_S}(M)\simeq\sO_S$, we will show that $M$ is in the image of $(1-\sigma^*)$. Using \eqref{eq:pullbackofnorm}, we know that $M\otimes\sigma^*M\simeq\sO_X$. By pulling back via $\phi$ we get that $\phi^*M\otimes\tilde{\sigma}^*\phi^*M$ is again trivial and by the proof of \ref{lem:potentialclasses} we see that $c_1(\phi^*M)=l(0,0,h\cdot\psi_1+k\cdot\psi_2)$ for two integers $h$ and $k$. Then we can write 
$$\phi^*M\simeq (1\times h\cdot\psi_1)^*\sP_A\otimes(1\times k\cdot \psi_2)^*\sP_A\otimes \gamma\simeq\phi^*(L_1^{\otimes h}\otimes L_2^{\otimes k})\otimes\gamma,
$$
for some $\gamma$ in $\Pic^0(A\times B)$. Therefore $\phi^*(M\otimes L_1^{\otimes -h}\otimes L_2^{\otimes -k})\simeq \gamma$, and we deduce that 
$M\simeq L_1^{\otimes h}\otimes L_2^{\otimes k}\otimes\alpha$ for some $\alpha\in\Pic^0(X)$. If $h$ and $k$ are both even, then by Lemma \ref{lem:type2im} we know that $\alpha\in \operatorname{Ker}\operatorname{Nm}_{\pi_S}$, and the class of $M$ in $\operatorname{Ker}\operatorname{Nm}_{\pi_S}/\operatorname{Im}(1-\sigma^*)$ is exactly $[\alpha]$. We apply Proposition \ref{prop:brabpic0} and deduce that $[M]=0$.\par
We will now show that neither one between $h$ and $k$ can be odd. Suppose otherwise that $h$ and $k$ are not both even. For example, assume that  $h$ is odd and $k$ is even, the proof in the other cases is very similar. Under this hypothesis, Lemma \ref{lem:type2im} ensures that $L_1\otimes\alpha$ is in the kernel of the norm map. But this contradicts Lemma \ref{lem:notker}, and our proof is complete
\end{proof}
\begin{ex}\label{last}
\begin{enumerate}[leftmargin=0cm,itemindent=1cm,labelwidth=.5cm,labelsep=.1cm,align=left,label=(\alph*)]
    \item Suppose that $A\simeq B$, then the isogenies $\psi_1$ and $\psi_2$ are indeed isomorphisms and thus the Brauer map can never be injective.
    \item Let $B$ be an elliptic curve without complex multiplication and consider $\theta_2$ a point of order 2 in $B$. Let $A$ be the elliptic curve $B/<\theta_2>$ and $\psi:B\rightarrow A$ the quotient map. The dual map $\psi^*$ has degree 2. Let $\theta_1\in A$ the point such that $\psi^*P_{\theta_1}$ is trivial and let $\tau$ be another order 2 element of $A$. All this data identify a bielliptic surface of type 2 whose Brauer map to the canonical cover is injective.
\end{enumerate}
\end{ex}
\newpage
\appendix
\section{The homomorphism lattice of two elliptic curves}
\pagestyle{headings}
\markleft{{JONAS BERGSTR\"OM AND SOFIA TIRABASSI}}
\markright{STRUCTURE OF $\operatorname{Hom}(A,B)$}
\begin{center}
\begin{small}
    {\scshape Jonas Bergstr\"om and Sofia Tirabassi}\\[\baselineskip]
\end{small}
\end{center}

The main goal of this appendix is to give a structure theorem for the $\mathbb{Z}$-module $\operatorname{Hom}(B,A)$ where $A$ and $B$ are two complex elliptic curves with $j(B)=0,1728$. This result has been used in \ref{subsec:NUMprod}
 above in order to make a clever choice of generators for $\operatorname{Num}(A\times B)$ which in turn has allowed an accurate description of the action of the automorphism $\sigma^*$ on the Neron--Severi group of the product $A\times B$ when $S$ is a bielliptic surface of type 3 or 5.\par
 If $B$ is an elliptic curve with $j$-invariant 0 or 1728, then $B$ admits an automorphism $\lambda_B$ of order 3 or 4 respectively. The main result of this Appendix is that the group $\operatorname{Hom}(B,A)$ can be completely described in terms of $\lambda_B$ and an isogeny $\psi:B\rightarrow A$. More precisely we have the following statement:
 \begin{thm}\label{thm:appendix2}
Let $A$ and $B$ two isogenous complex elliptic curves with $j(B)$ is either 0 or 1728. Then there exist an isogeny $\psi:B\rightarrow A$ such that 
$$\operatorname{Hom}(B,A)=<\psi,\psi\circ\lambda_B>.$$
\end{thm}

This appendix is organized in three main parts. In the first we outline some classical results about imaginary quadratic fields and their orders. The second is concerned with complex elliptic curves with complex multiplication. Theorem \ref{thm:appendix2}
 is proven in \ref{appendixmain}. The key idea of our argument is to describe $\operatorname{Hom}(B,A)$ as a fractional ideal of $\operatorname{End}(B)$ homothetic to $\operatorname{End}(B)$. This is done by observing that the class number of $\operatorname{End}(B)$ is 1.
 
 \subsection{Preliminaries on orders in imaginary quadratic fields}
An \emph{imaginary quadratic field} is a subfield $K\subseteq \mathbb{C}$ of the form $\mathbb{Q}(\sqrt{-d})$, with $d$ a positive, square-free integer. The \emph{discriminant of $K$} is  the integer $d_k$ defined as
\begin{equation*}
d_K =\begin{cases}
-d, &\text{if $d\equiv 3\mod 4$},\\
-4d, &\text{otherwise}.
\end{cases}
\end{equation*}
The \emph{ring of integers of $K$}, $\mathcal{O}_K$ is the largest subring of $K$ which is a finitely generated abelian group. Then we have that $\mathcal{O}_K=\mathbb{Z}[\delta]$, where 
\begin{equation}\label{eq;DELTA}
\delta =\begin{cases}
\frac{1+\sqrt{-d}}{2}, &\text{if $d\equiv 3\mod 4$},\\
\sqrt{-d}, &\text{otherwise}.
\end{cases}
\end{equation}
An \emph{order} in an imaginary quadratic field $K$ is a subring $\mathcal{O}$ of $\mathcal{O}_K$ which properly contains $\Z$. It turns out that $\mathcal{O}\simeq \mathbb{Z}+\Z\cdot (n\delta)$ for some positive integer $n$.\par
Given an order $\mathcal{O}$ in an imaginary quadratic field $K$, a \emph{fractional ideal} of $\mathcal{O}$ is a non-zero finitely generated sub $\mathcal{O}$-module of $K$. For every  fractional ideal $M$ of $\mathcal{O}$ there is an $\alpha\in K^*$ and an ideal $\mathfrak{a}$ of $\mathcal{O}$ such that $M=\alpha\cdot \mathfrak{a}$.
We will need the following notions.
\begin{defi}
\begin{enumerate}[leftmargin=0cm,itemindent=1cm,labelwidth=.5cm,labelsep=.1cm,align=left,label=(\roman*)]
    \item Two fractional $\mathcal{O}$-ideals $M$ and $M'$ are \emph{homothetic} if there is $\alpha\in K^*$ such that $M=\alpha M'$.
    \item A fractional $\mathcal{O}$-ideal is \emph{invertible} if there is a fractional ideal $M'$ such that $M\cdot M'=\mathcal{O}$. The set of invertible $\mathcal{O}$-ideals is denoted by $I(\mathcal{O})$.
    \item A fractional $\mathcal{O}$-ideal $M$ is principal if it is of the form $\alpha\cdot\mathcal{O}$ for some $\alpha\in K^*$. So principal ideals are precisely the fractional ideals homothetic to $\mathcal{O}$. The set of principal $\mathcal{O}$-ideals is denoted by $P(\mathcal{O})$.
\end{enumerate}
\end{defi}
Principal ideals are clearly invertible. In general not all fractional ideals are invertible, but they are so if $\mathcal{O}=\mathcal{O}_K$ (see also \cite[Proposition 5.7]{cox2011primes}). 
The quotient 
$$
\mathfrak{Cl}(\mathcal{O}):=I(\mathcal{O})/P(\mathcal{O})
$$
describes the homothety classes of invertible $\mathcal{O}$-ideals. It is a group with the product and it is called the \emph{ideal class group of $\mathcal{O}$}. Its order is finite and is called the \emph{class number of $\mathcal{O}$}. When $\mathcal{O}=\mathcal{O}_K$, then the class number of $\mathcal{O}$ is exactly the class number of the field $K$, which is a function of the discriminant of $K$ (see \cite[Theorem 5.30(ii)]{cox2011primes}). More generally the class number of $\mathcal{O}$ is a general function of $d_K$ and $[\mathcal{O}:\mathcal{O}_K]$.
\begin{ex}\label{classnm}
If $K$ is either $\mathbb{Q}(i)$ or $\mathbb{Q}(\sqrt{-3})$, then all the fractional ideals of $\mathcal{O}_K$ are homothetic to $\mathcal{O}_K$. In fact the class number of the field $K$ in this cases is 1, as it was computed by Gauss in his book \emph{Disquisitiones arithmeticae} \nocite{gauss1966disquisitiones}.
\end{ex}
\subsection{Elliptic curves with complex multiplication}\label{sub:CM}
 The importance of orders in the study of the geometry of elliptic curves is that they describe the endomorphism rings of a complex elliptic curves:
\begin{thm}
Let $A$ be an elliptic curve over $\mathbb{C}$, then $\End(A)$ is  either isomorphic to $\Z$ or to an order in an imaginary quadratic field.
\end{thm}
\begin{proof}
See \cite[Theorem VI.5.5]{Sil2009}.
\end{proof}
We say that a (complex) elliptic curve has \emph{complex multiplication} if its endomorphisms ring is larger than $\mathbb{Z}$. Observe that in this case $\End(A)\otimes\mathbb{Q}$ is a quadratic field $K$ and $\End(A)$ is an order in $K$.\par
Given a complex elliptic curve $A$ there is a canonical way to identify its endomorphisms ring with a subring of $\mathbb{C}$. More generally let $A$ and $B$ two elliptic curves, then there are two lattices $\Lambda_A$ and $\Lambda_B$ in $\mathbb{C}$ such that $A\simeq\mathbb{C}/\Lambda_A$ and $B\simeq\mathbb{C}/\Lambda_B$. Given a complex number $\zeta$ such that $\zeta\cdot\Lambda_B\subseteq \Lambda_A$, the map $\Phi_\zeta:\mathbb{C}\rightarrow\mathbb{C}$ defined by $z\mapsto \zeta\cdot z$ descends to an (algebraic) homomorphism $\varphi_\zeta:B\rightarrow A$. It is possible to show (see \cite[VI.5.3(d)]{Sil2009}) that any morphism of elliptic curves preserving the origin is obtained in this way, and in particular we get an isomorphism of abelian groups
\begin{equation}\label{eq:homab}
    \Hom(B,A)\simeq\{\zeta\in\mathbb{C}\:|\:\zeta\cdot\Lambda_B\subseteq\Lambda_A\}\subseteq\mathbb{C}.
\end{equation}

By setting $B=A$ we get a ring isomorphism
$$
\End(A)\simeq \mathcal{O}:=\{\zeta\in\mathbb{C}\:|\:\zeta\cdot\Lambda_A\subseteq\Lambda_A\}\subseteq\mathbb{C}.
$$
The isomorphism $\zeta\mapsto\varphi_\zeta$ is characterized as the unique isomorphism $f:\mathcal{O}\rightarrow\End(A)$ such that, for any $\zeta\in \mathcal{O}$ and for every invariant form $\omega$ on $A$ we have that $f(\zeta)^*\omega=\zeta\cdot \omega$ (\cite[II.1.1]{Sil2013}).\par
\begin{nota}\label{not:lambdaA}
For an elliptic curve with complex multiplication $A$ such that $\End(A)\simeq \mathbb{Z}+\mathbb{Z}\cdot n\delta$, we will denote by $\lambda_A$ the isogeny $\varphi_{n\delta}:A\rightarrow A$ and we will say that \emph{$A$ has complex multiplication by $\lambda_A$}. 
\end{nota}
It is clear that, with this identification, $\End(A)=<1_A,\lambda_A>$ as a $\Z$-module.
\begin{ex}\label{ex:j=01728}
\begin{enumerate}[leftmargin=0cm,itemindent=1cm,labelwidth=.5cm,labelsep=.1cm,align=left,label=(\alph*)]
    \item Suppose that $B$ is an elliptic curve with $j$-invariant 0. Then we can write $B\simeq \mathbb{C}/\Lambda_B$, with $\Lambda_B=<1, e^{\frac{2\pi i}{3}}>$. Then $\End(B)\otimes\mathbb{Q}\simeq\mathbb{Q}(\sqrt{-3})$ and $\End(B)\simeq\mathcal{O}_K=\Z[\frac{1+\sqrt{-3}}{2}]$. We have that $\lambda_B$ is induced by the multiplication by $\frac{1+\sqrt{-3}}{2}$ and is an automorphism of $B$ satisfying $\lambda_B^2+\lambda_B+1_B=0$. This is exactly the automorphism which in \ref{sub:cancover}  was denoted by $\rho$ and which was used to construct bielliptic surfaces of type 5. 
    \item Suppose now that the $j$-invariant of $B$ is 1728. Then we can take $\Lambda_B=<1,i>$ and we have that $\End(B)\otimes\mathbb{Q}\simeq \mathbb{Q}(i)$. The endomorphisms ring of $B$ is isomorphic to $\mathbb{Z}[i]$ and the multiplication by $i$  induces an automorphism $\lambda_B$ such that $\lambda_B^2=-1_B$. This is the automorphism $\omega$ of $B$ used to construct bielliptic surfaces of type 3 in \ref{sub:cancover}.
\end{enumerate}
\end{ex}
\subsection{Proof of the main result}\label{appendixmain}
We are now ready to provide a proof for Theorem \ref{thm:appendix2}. Our key point will be the following:\\[\baselineskip]
\emph{Claim:} the $\mathbb{Z}$-module $\operatorname{Hom}(B,A)$ is isomorphic to a fractional ideal of $\mathcal{O}_K$.\\[\baselineskip]
Before proceeding with showing that this Claim is true, let us see how it implies the statement. We do this applying Example \ref{classnm} and deducing that all fractional $\mathcal{O}_K$-ideals are homothetic to $\mathcal{O}_K$. Therefore there exist $\alpha\in K^* $ such that 
$$
M\simeq\alpha\cdot\mathcal{O}_K=\alpha\cdot<1,\delta> =<\alpha,\alpha\cdot\delta>,
$$
where $\delta$ is like in \eqref{eq;DELTA}. But then we have that $\operatorname{Hom}(B,A)=<\varphi_\alpha,\varphi_\alpha\circ\lambda_B>$, and the statement is true.
\begin{proof}[Proof of the Claim]
Let $\Lambda_A=<1,\tau>$ a lattice in $\mathbb{C}$ such that $A\simeq\mathbb{C}/\Lambda_A$, and denote by $K\subseteq \mathbb{C}$ the quadratic field $\operatorname{End}(B)\otimes \mathbb{Q}
$. Then the ring $\operatorname{End}(B)$ is exactly the ring of integers $\mathcal{O}_K$. Observe that this is isomorphic to a lattice in  $\mathbb{C}$, and that $B\simeq \mathbb{C}/\mathcal{O}_K$ (See Example \ref{ex:j=01728}). \par
By \eqref{eq:homab} we can identify  $M:=\operatorname{Hom}(B,A)$ as a finitely generated subgroup of $\mathbb{C}$.  Composition on the right with endomorphism of $B$ gives to $M$ a structure of $\mathcal{O}_K$-module. 
Let $\alpha\neq 0$ denote an element of $\mathfrak{a}:=\operatorname{Hom}(A,B)$, identifyied with a subgroup of $\mathbb{C}$. Then clearly $\alpha\cdot M\subseteq \mathcal{O}_K$. We deduce that $M\subseteq K$ is a fractional ideal of $\mathcal{O}_K$, and the Claim is proven.
\end{proof}
\begin{rmk}
\begin{enumerate}[leftmargin=0cm,itemindent=1cm,labelwidth=.5cm,labelsep=.1cm,align=left,label=(\alph*)]
    
    \item For any order $\mathcal{O}$ in a quadratic extension of $\mathbb{Q}$ a representative of each homothety class of fractional ideals can be given as $I\cdot \mathcal{O}'$, where $\mathcal{O'}\supseteq\mathcal{O}$ is an over order and $I$ is an invertible fractional ideal (see \cite{marseglia2018computing}). The over order $S$ can be given a $\Z$-basis of the form $\{1,\delta\cdot f\}$ where $f$ is a positive integer.\par
    For any pair of isogenous complex elliptic curves with complex multiplication we have that $\operatorname{Hom}(B,A)$ is a fractional $\operatorname{End}(B)$ ideal. In addition, if we assume that $\operatorname{End}(B)$ has class number 1, we have that $B\simeq \mathbb{C}/\operatorname{End}(B)$. In fact, under this assumption  \cite[Corollary 10.20]{cox2011primes} yields that, there is just one elliptic curve up to isomorphism with endomorphism ring $\operatorname{End(B)}$.\par
    In conclusion, demanding that $\operatorname{End}(B)$ has class number 1 (instead of $j(B)$ being either 0 or 1728) is sufficient for Theorem \ref{appendixmain} to hold.\par
    So Theorem \ref{appendixmain} will hold for the 13 isomorphism classes of complex elliptic curves $B$ for which $\operatorname{End}(B)$ has discriminant -3, -4,-7, -8, -11, -12, -16, -19, -27, -28, -43, -67, -164 (see \cite[Example 11.3.2]{Sil2013}).
    \item It is clear from the proof that the role of $A$ and $B$ can be exchanged, so we have proven a structure theorem for $\operatorname{Hom}(A,B)$ when the endomorphism ring of one of the two curves has class number 1.
\end{enumerate}
\end{rmk}
Theorem \ref{appendixmain} is not constructive, in the sense that it does not provide a way to determine the isogeny $\psi$ such that $\psi$ and $\psi\circ\lambda_B$ generate $\operatorname{Hom}(B,A)$. We conclude this appendix by constructing $\psi$ in an easy example.
\begin{ex}\label{ex:computation}
Let $\Lambda$ be the lattice $<1,2i>\subseteq \mathbb{C}$, and consider $A:=\mathbb{C}/\Lambda.$ Consider the 2-torsion point $\tau:=(0,i)+\Lambda$ of $A$ and let $B$ be the quotient $A/<\tau>$. It is clear that $B$ has $j$-invariant 1728. We claim that $\operatorname{Hom}(B,A)=<\varphi_2,\varphi_2\circ\lambda_B>$.\par
We use first \eqref{eq:homab} and identify $\operatorname{Hom}(B,A)$ with a lattice in $\mathbb{C}$. Given $\alpha=(a+bi)\in\operatorname{Hom}(B,A),$
we have that both $\alpha$ and $\alpha\cdot i$ must be elements of $\Lambda$. We deduce that both $a$ and $b$ must be even integers and so $\operatorname{Hom}(B,A)=<2,2\cdot i>$. We conclude by observing that $\lambda_B$ is the automorphism of $B$ induced by multiplication by $i$.
\end{ex}

\bibliographystyle{alpha}
\bibliography{biblio}
\end{document}